\documentclass{jms_My_en}
\paperTitle{The calculation of the probability density of a strictly stable law at large~$X$}
\articleColonName{The calculation of a strictly stable law at $x\to\infty$}
\authorsShort{V.\,V.\,Saenko}
\authorsFull{V.\,V. Saenko\first}
\addAuthorInfo{Ulyanovsk State University, S.P. Kapitsa Research Institute of Technology, Ulyanovsk, 42, Leo Tolstoy St.,  432017 email: \url{vvsaenko@inbox.ru}}
\usepackage{color}

\paperAbstract{The article is devoted to the problem of calculating the probability density of a strictly stable law at $x\to\infty$. To solve this problem, it was proposed to use the expansion of the probability density in a power series. A representation of the probability density in the form of a power series and an estimate for the remainder term was obtained. This power series is convergent in the case $0<\alpha<1$ and asymptotic at $x\to\infty$ in the case $1<\alpha<2$. The case $\alpha=1$ was considered separately. It was shown that in the case $\alpha=1$ the obtained power series was convergent for any $|x|>1$ at $N\to\infty$. It was also shown that in this case it was convergent to the density of $g(x,1,\theta)$. An estimate of the threshold coordinate $x_\varepsilon^N$, was obtained which determines the range of applicability of the resulting expansion of the probability density in a power series. It was shown that in the domain $|x|\geqslant x_\varepsilon^N$ this power series could be used to calculate the probability density.}
\begin{document}

\maketitle

\section{Introduction}
The probability density and distribution function are important characteristics of any probability distribution. The possibility to calculate these quantities gives an opportunity to solve a wide range of probabilistic and statistical problems, ranging from the calculation of the probability of accepting a particular value for some random variable and ending with estimates of distribution parameters from experimental data. The absence of explicit expressions of these characteristics, except in a few cases, makes the class of stable laws inconvenient to use. The only exception is the five cases where the density of the stable law is expressed in terms of elementary functions: the Levy distribution   ($\alpha=1/2,\theta=1$), the symmetric Levy distribution ($\alpha=1/2,\theta=-1$), the Cauchy distribution  ($\alpha=1,\theta=0$), the Gaussian distribution ($\alpha=2,\theta=0$) and  the generalized Cauchy distribution  ($\alpha=1,-1\leqslant\theta\leqslant1$). The last distribution was first written out in the monograph by V.M. Zolotarev (see~\cite{Zolotarev1986} \S2.3 formula (2.3.5a)) and later it was studied in the works \cite{Saenko2020c,Saenko2020b}.  Here $\alpha$ is the characteristic parameter of the stable law  ($0<\alpha\leqslant2$), $\theta$  is the asymmetry parameter ($-1\leqslant\theta\leqslant1$). For other values of the parameters $\alpha$ and $\theta$ it is necessary to use different representations for the probability density of stable laws.

The most common method for calculating the probability density is based on the use of the integral representation. The inverse Fourier transform of the characteristic function of the stable law makes it possible to obtain two types of integral representations. The integral representations that express the probability density in terms of an improper integral of an oscillating function are referred to the first type. The integral representations expressing the probability density in terms of a definite integral of a monotonic function belong to the second type. The second method of the inverse Fourier transform is called the stationary phase method. Each of these two types of integral representations has both advantages and disadvantages.

Integral representations related to the first type were obtained and studied in the works \cite{Nolan1999,Ament2018}, for the characteristic function in the parameterization <<M>>. Here and further in the text, the designation of various parametrizations of the characteristic function of a stable law corresponds to the designations introduced in the book by V.M. Zolotarev \cite{Zolotarev1986}. The main difficulty in using this representation lies in the oscillating integrand. This behavior of the function leads to calculation difficulties for numerical integration algorithms. The work \cite{Nolan1999} indicates the following limitations when using this representation: 1)~in the case $\alpha<0.75$ the integration domain becomes very large, which leads to difficulties in numerical integration; 2)~if $\beta\neq0$ и $0<|\alpha-1|<0.001$ there are problems in calculating the term with  $(\tan(\pi\alpha/2)(t-t^\alpha)$; 3)~when $x$ is very large, the integrand oscillates very quickly.  To eliminate these difficulties, in the article \cite{Ament2018} the standard quadrature method of numerical integration is modernized to adapt it for calculating integrals of oscillating functions. This made it possible to reduce the lower limit of the parameter $\alpha$ from the value  0.75 to the value   0.5. To calculate the probability density for large $x$ in the article \cite{Ament2018} there is a suggestion of using the corresponding representation for the probability density in the form of a power series. This approach makes it possible to solve the problem of calculating the density for large $x$ completely.  At values $x\to\zeta$ in the case $\alpha>0.5$ the integrand behaves quite well, and therefore there are no problems with the calculation of the integral. However, the article points out that the proposed scheme is not applicable for symmetric distributions in the case of $\alpha<0.5$ and for non-asymmetric distributions in the cases $\alpha<0.5$ and $0.9<\alpha<1.1$.

The second type of integral representations for the probability density of stable laws is widely used. These integral representations express the density in terms of a definite integral of a monotonic function. This behavior of the integrand turns out to be more convenient for practical purposes. For the first time, such an integral representation was obtained in the article \cite{Zolotarev1964_en}. It expresses the probability density of the stable law for the characteristic function in the parameterization <<B>> through a definite integral. Later, the results of this article were included in the monographs \cite{Zolotarev1986} (see~\S2.2) and \cite{Uchaikin1999} (see~\S4.4, \S4.5). For the parameterization<<M>>, a similar representation was obtained in the articles \cite{Nolan1997, Nolan2022}, and for the parameterization <<C>> in  \cite{Saenko2020b}. Since the integrands of these representations are monotonic, there are no problems with the calculation of these integrals in a wide range of coordinates. In this regard, these integral representations are widely used in problems of calculating the probability density and the distribution function of stable laws. The integral representation for the parameterization <<M>> has gained in special popularity. Several software products have been developed based on this idea \cite{Liang2013,Royuela-del-Val2017,Julian-Moreno2017,Rimmer2005,Veillette2008}.

In a wide range of values of the coordinate $x$ there are no problems with the calculation of the density using these integral representations. However, there are domains of the coordinate $x$, where numerical integration algorithms cannot calculate these integrals. This is pointed out in the articles \cite{Nolan1997,Pogany2015,Royuela-del-Val2017,Ament2018}. The reason for these difficulties lies in the behavior of the integrand. For very small or large values of $x$ the integrand has the form of a very narrow peak. As a result, numerical integration algorithms cannot recognize it and start producing an incorrect result. Various numerical algorithms are proposed in the articles \cite{Nolan1997,Royuela-del-Val2017,Julian-Moreno2017} to solve this problem. However, all these algorithms increase the accuracy of calculations, but do not eliminate the problem completely. To solve the problem of calculating the probability density for small or large values of $x$ it is expedient to use such representations for the density that do not have any singularities in these areas. Power series representations for the probability density are the most suitable option for this purpose.

These representations are obtained by expanding the probability density in a power series in the vicinity of the points $x=0$ and $x=\infty$.  Such representations are well known. One of the first papers in which such an expansion was obtained is the work \cite{Pollard1946}. In this work, the expansion of the probability density into a convergent series is obtained in the case of $x\to\infty$ and $0<\alpha<1$. In the article \cite{Bergstrom1952} a generalization of this expansion was obtained at $x\to\infty$ in the case $1<\alpha<2$. In this range of values of the parameter $\alpha$ this series turns out to be asymptotic. In the same article, the expansion of the density in a series was obtained in the vicinity of the point $x\to0$ for the case $0<\alpha<2$.  The obtained power series is asymptotic in the case of $0<\alpha<1$, and it is convergent in the case of $1<\alpha<2$. The expansions for $\alpha>1$ in the cases $x\to0$ and $x\to\infty$ were also obtained in the paper \cite{Schneider1986} as a result of the expansion of the probability density into a power series expressed in terms of the Fox function. All these expansions can be found in the books \cite{Feller1971_V2_en} (see Chapter 17, \S7), \cite{Zolotarev1986} (see \S2.4 and \S2.5) and \cite{Uchaikin1999} (see \S4.4, \S4.5). We should note that all these results refer to stable laws with a characteristic function in the parameterization <<B>>.  For the parametrization <<M>>, expansions in a power series were obtained in the article \cite{Ament2018}. For the parameterization <<C>>, expansions are given in the works \cite{Bening2006,Saenko2022b}. A characteristic property of all these expansions is that, depending on the value of the exponent $\alpha$ power series are either convergent or asymptotic. It should be noted that there is a way of obtaining representations in the form of convergent power series for all admissible values of the parameter $\alpha$. This method is described in the article \cite{Arias-Calluari2017} and assumes the use of trans-stable laws. Using this approach, expansions in a power series for the probability density of a symmetric stable law at $x\to0$ and $x\to\infty$ were obtained in this article for the cases $0<\alpha<1$ and $1<\alpha<2$. At the same time one thing should be pointed out is that these series are convergent.

This article deals with the problem of calculating the probability density of a strictly stable law with a characteristic function
\begin{equation}\label{eq:CF_formC}
  \hat{g}(t,\alpha,\theta,\lambda)=\exp\left\{-\lambda |t|^\alpha\exp\{-i\tfrac{\pi}{2}\alpha\theta\sign t\}\right\},\quad t\in\mathbf{R}
\end{equation}
in the case $x\to\infty$. Here $\alpha\in(0,2]$, $|\theta|\leqslant\min(1,2/\alpha-1)$, $\lambda>0$. This characteristic function corresponds to the parameterization <<C>>. In the article \cite{Saenko2020b} using the stationary phase method, the inverse Fourier transform of this characteristic function was performed and integral representations for the probability density and distribution function (see~Appendix~\ref{sec:IntRepr}) were obtained. As we can see the formula (\ref{eq:g(x)_int}) expresses the probability density in terms of a definite integral. Moreover, the integrand is a monotonic function and behaves quite well in a large range of values of $x$.  Therefore, in a fairly wide range of values of $x$ there are no difficulties with the numerical calculation of this integral. However, for very small and very large values of $x$ the integrand in (\ref{eq:g(x)_int}) takes the form of a very narrow peak. As a result, numerical integration algorithms cannot recognize this peak and produce an incorrect integration result. This article will consider the problem of calculating the probability density at $x\to\infty$. To solve this problem, the expansion of the probability density in a series at $x\to\infty$ will be obtained and the range of applicability of this series will be examined.  It should be noted that the problem of calculating the probability density in the case of $x\to0$ will not be considered here, since it was considered earlier in the article \cite{Saenko2022b}.

\section{The representation of the probability density in the form of  a power series}

We will undertake the task of obtaining the expansion of the probability density in a power series in the case of $x\to\infty$. In the Introduction, it was pointed out that similar expansions for stable laws with a characteristic function in the parameterization <<B>> \cite{Pollard1946,Bergstrom1952,Schneider1986,Zolotarev1986,Feller1971_V2_en}  and <<M>> \cite{Ament2018} were obtained earlier and are well known. Depending on the value of the characteristic exponent $\alpha$ these power series turn out to be either convergent or asymptotic. In the case considered here ($x\to\infty$) at the values $0<\alpha<1$ the representation of the probability density in the form of a power series turns out to converge at $N\to\infty$, and the values  $1<\alpha<2$ are asymptotic at $x\to\infty$. In this regard, when deriving these expansions, the cases of $\alpha<1$ and $\alpha>1$ are considered separately.  In addition, the existing evidence for obtaining the expansion of the density of the stable law in a power series at $x\to\infty$ contain some flaws, without correcting which it is impossible to provide a rigorous proof. Therefore, we obtain the expansion for the probability density in the case of $x\to\infty$, along the way correcting some flaws in the proofs. We will initially consider the entire range of values of the characteristic index $\alpha$. This will give an opportunity to show that the expansion obtained is valid for all admissible $\alpha$. Further, it will be shown that they are convergent at $0<\alpha<1$ and asymptotic at $1<\alpha<2$.

Without loss of generality, further we will assume everywhere that $\lambda=1$. Strictly stable laws with the parameter $\lambda=1$ are usually called standard strict stable laws. Abbreviations have been adopted for these laws. The characteristic function is usually denoted $\hat{g}(t,\alpha,\theta,1)\equiv\hat{g}(t,\alpha,\theta)$, the probability density - $g(x,\alpha,\theta,1)\equiv g(x,\alpha,\theta)$, a strictly stable random variable $Y(\alpha,\theta,1)\equiv Y(\alpha,\theta)$. Further, we will use these notations for standard strictly stable quantities. It should be noted that to transform a standard strictly stable law into a strictly stable law with an arbitrary $\lambda$ one can use the results of the article \cite{Saenko2020b} (see also \cite{Zolotarev1986, Uchaikin1999}).

To obtain the expansion, we use the formula for inverting the characteristic function and, then, expand the integrand into a power series. This will give an opportunity to calculate the integral and get the expansion of the probability density in a power series. To invert the characteristic function, we use  lemma~\ref{lem:Inverse} (see Appendix~\ref{sec:IntRepr}). This lemma formulates two formulas for the inversion of the characteristic function. In principle, there is no big difference which of these two formulas to use. The result will differ only by the sign of the parameter $\theta$. In this article we will use the first formula (\ref{eq:InverseFormula}). Thus,
\begin{equation*}
  g(x,\alpha,\theta)=\frac{1}{\pi}\Re\int_{0}^{\infty}e^{itx}\hat{g}(t,\alpha,-\theta).
\end{equation*}
We also make use of the inversion property which can be formulated for a stable law with a characteristic function (\ref{eq:CF_formC}) in the form
\begin{property}\label{prop:Inversion}
For any admissible parameters $(\alpha,\theta)$
\begin{equation*}
  Y(\alpha,-\theta)\stackrel{d}{=}-Y(\alpha,\theta).
\end{equation*}
\end{property}
The proof of this property is given in the article \cite{Saenko2020b} (see also \cite{Uchaikin1999, Zolotarev1986}).   In terms of the characteristic function $\hat{g}(t,\alpha,\theta)$ and probability density $g(x,\alpha,\theta)$  this property takes the form
\begin{equation}\label{eq:InversionFormula}
  \hat{g}(-t,\alpha,\theta)=\hat{g}(t,\alpha,-\theta),\quad g(-x,\alpha,\theta)=g(x,\alpha,-\theta).
\end{equation}
The convenience of this property lies in the fact that, when studying the probability density, it gives us an opportunity to confine ourselves to considering only the case $x\geqslant0$. Expressions for the case $x<0$ are obtained with the use of the expressions mentioned above.

Next, we need an analytic continuation of the function (\ref{eq:CF_formC}) to the complex plane $z=t e^{i\varphi}$. In the book \cite{Zolotarev1986} (see\S~2.2) it was shown that such an analytic continuation can be performed with the half-line $t>0$ (or the half-line $t<0$) with a cut along some contour connecting the points  0 and $\infty$. We perform analytic continuation with the half-line $t>0$ with a cut along the negative part of the real half-line $\arg z=\pi$. We will denote the analytic continuation thus obtained
\begin{equation*}
  g^{+}(z,\alpha,\theta)=\exp\left\{-z^\alpha\exp\left\{-\tfrac{\pi}{2}\alpha\theta\right\}\right\}.
\end{equation*}

We will start our study with consideration of the integral
\begin{equation}\label{eq:I_Gamma}
  I(\Gamma)=\frac{1}{x\pi}\int_{\Gamma} e^{i\zeta}g^{+}(\zeta/x,\alpha,-\theta)d\zeta
\end{equation}
The following lemma is valid for this integral
\begin{lemma}\label{lem:I_GammaR}
For any admissible values of the parameters $\alpha$ and $\theta$ and for any positive $\epsilon\to0$ in the case $x\to\infty$,
if the conditions are met:
\begin{enumerate}
  \item $0<\alpha<1$ and $-1\leqslant\theta\leqslant1$ the contour $\Gamma_R$ has the form $\Gamma_R=\left\{\zeta:\ \tau=R,\ 0\leqslant\varphi\leqslant\pi-\epsilon \right\}$, where $R\to\infty$ in such a way that $R/x\to\infty$;
  \item $\alpha=1$ and $-1<\theta<1$, the contour $\Gamma_R$ has the form $\Gamma_R=\left\{\zeta:\ \tau=R,\ 0\leqslant\varphi\leqslant\pi-\epsilon \right\}$, where $R\to\infty$ in such a way that $R/x\to\infty$;
  \item $1<\alpha\leqslant2$, the contour $\Gamma_R$ has the form $\Gamma_R=\left\{\zeta:\ \tau=R,\ -\frac{\pi}{2\alpha}-\frac{\pi}{2}\theta+\epsilon\leqslant\varphi\leqslant
      \frac{\pi}{2\alpha}-\frac{\pi}{2}\theta\right\}$, where $R\to\infty$ in such a way that  $R/x\to\infty$;
  \item $0<\alpha\leqslant2$, the contour $\Gamma_R$ has the form $\Gamma_R=\left\{\zeta:\ \tau=R,\ 0+\epsilon\leqslant\varphi\leqslant\pi-\epsilon\right\}$, where $R\to\infty$ in such a way that   $R/x\approx 1$;
  \item $0<\alpha\leqslant2$, the contour $\Gamma_R$ has the form $\Gamma_R=\left\{\zeta:\ \tau=R,\ 0\leqslant\varphi\leqslant\pi \right\}$, where $R\to\infty$ in such a way that  $R/x\to0$;
  \item $0<\alpha\leqslant2$, the contour $\Gamma_R$ has the form $\Gamma_R=\left\{\zeta:\ \tau=R,\ -\pi\leqslant\varphi\leqslant\pi\right\}$, where $R\to0$,
\end{enumerate}
then
\begin{equation*}
  I(\Gamma_R)=\frac{1}{\pi x}\int_{\Gamma_R} e^{i\zeta}g^{+}(\zeta/x,\alpha,-\theta)d\zeta\to0.
\end{equation*}
Here $\zeta=\tau e^{i\varphi}$,
\end{lemma}

\begin{proof} 
We will consider the integral (\ref{eq:I_Gamma}) and choose in the complex plane  $\zeta=\tau e^{i\varphi}$. As an integration contour we will choose a circle arc of radius $R$
\begin{equation*}
  \Gamma_R=\left\{\zeta:\ \tau=R,\ \varphi_1\leqslant\varphi\leqslant\varphi_2\right\}.
\end{equation*}
For any contour $\Gamma_R$ the inequality is valid
\begin{equation*}
  I(\Gamma_R)\leqslant|I(\Gamma_R)|\leqslant\frac{1}{\pi x}\int_{\Gamma_R}|e^{i\zeta}g^+(\zeta/x,\alpha,-\theta)d\zeta|.
\end{equation*}
Assuming $\zeta=\tau e^{i\varphi}$, we get
\begin{multline}\label{eq:I_GammaR}
  I(\Gamma_R)\leqslant \frac{1}{\pi x}\int_{\varphi_1}^{\varphi_2}\left|e^{iR\exp\{i\varphi\}} g^+\left(\frac{R}{x}e^{i\varphi},\alpha,-\theta\right)iRe^{i\varphi}d\varphi\right| \\
  =\frac{1}{\pi x}\int_{\varphi_1}^{\varphi_2} \left|R\exp\left\{iRe^{i\varphi}- (R/x)^\alpha \exp\left\{i\alpha(\varphi+\tfrac{\pi}{2}\theta)\right\}+i\varphi\right\}\right| d\varphi
  = \int_{\varphi_1}^{\varphi_2}V(\varphi,R,x,\alpha,\theta)d\varphi,
\end{multline}
where the notation was introduced
\begin{equation}\label{eq:Vfun}
  V(\varphi,r,\xi,\alpha,\theta)=\frac{r}{\pi \xi}\exp\left\{-r\sin\varphi-(r/\xi)^\alpha\cos\left(\alpha\left(\varphi+\tfrac{\pi}{2}\theta\right)\right)\right\}.
\end{equation}

Our objective is to determine the conditions imposed on the contour $\Gamma_R$ and on the values of the parameters $\alpha,\theta$, at which the limits $\lim_{\stackrel{R\to\infty}{x\to\infty}}I(\Gamma_R)=0$ and $\lim_{\stackrel{R\to0}{x\to\infty}}I(\Gamma_R)=0$.   With the simultaneous passage of $R$ and $x$ to infinity, in view of  the arbitrariness in the choice of radius $R$ three situations can be realized: 1) $R\to\infty, x\to\infty$, in such a way that $R/x\to\infty$; 2) $R\to\infty, x\to\infty$ in such a way that $R/x\approx1$; 3) $R\to\infty, x\to\infty$ in such a way that $R/x\to0$. We will consider each of these cases separately. For this purpose we will perform in (\ref{eq:I_GammaR}) the passage to the limit $R\to\infty$, $x\to\infty$. As a result, we get
\begin{equation}\label{eq:I_gamR_lim}
  \lim_{\stackrel{R\to\infty}{x\to\infty}}I_{\Gamma_R}\leqslant\lim_{\stackrel{R\to\infty}{x\to\infty}}|I_{\Gamma_R}| \leqslant
  \lim_{\stackrel{R\to\infty}{x\to\infty}}\int_{\varphi_1}^{\varphi_2}V(\varphi,R,x,\alpha,\theta)d\varphi,
\end{equation}
Also, in each of the three cases described above, it is necessary to consider the situations separately when $0<\alpha\leqslant1$ and when $1<\alpha\leqslant2$. We will start examination with the first case.

\textit{The case $R\to\infty$, $x\to\infty$, $R/x\to\infty$}.
We will consider the case $0<\alpha\leqslant1$ first. From (\ref{eq:I_gamR_lim}) it is clear that the value of this limit is determined by the limit of the function $V(\varphi,R,x,\alpha,\theta)$ in the case $R\to\infty$ and $x\to\infty$.  Taking into account that $\alpha\leqslant1$, we obtain
\begin{multline}\label{eq:Ulim_a<1_rx2int}
  \lim_{\stackrel{R\to\infty}{x\to\infty}} V(\varphi,R,x,\alpha,\theta)=
  \lim_{\stackrel{R\to\infty}{x\to\infty}} \frac{R}{\pi x} \exp\left\{-R\left(\sin\varphi+ \frac{R^{\alpha-1}}{x^\alpha}\cos\left(\alpha\left(\varphi+\tfrac{\pi}{\theta}\right)\right)\right)\right\}\\
  = \lim_{\stackrel{R\to\infty}{x\to\infty}} \frac{R}{\pi x}e^{-R\sin\varphi}=0, \quad \mbox{if} \quad 0<\varphi<\pi.
\end{multline}
Here a suggestion was used that in the case $\alpha=1$ $R^0=1$. This limit also gives us the values of the integration limits $\varphi_1$ and $\varphi_2$ in (\ref{eq:I_gamR_lim}). From (\ref{eq:Ulim_a<1_rx2int}) it is clear that $\varphi_1\to0$, а $\varphi_2\to\pi$. This means that the limits of integration $\varphi_1$ and $\varphi_2$ approach the values $0$ and $\pi$ indefinitely but do not take these values. To determine if these integration limits can take the values $0$ and $\pi$ it is necessary to calculate the limit (\ref{eq:I_gamR_lim}) at these points.

We get down to this issue and consider the behavior of the function $V(\varphi,R,x,\alpha,\theta)$ at the points  $\varphi=0$ and $\varphi=\pi$. In the first case we have
\begin{equation*}
\lim_{\stackrel{R\to\infty}{x\to\infty}}V(0,R,x,\alpha,\theta)=
\lim_{\stackrel{R\to\infty}{x\to\infty}} \frac{R}{\pi x}\exp\left\{-(R/x)^\alpha\cos\left(\tfrac{\pi}{2}\alpha\theta\right)\right\}
\end{equation*}
Taking account that the considered case is $\alpha\leqslant1$ and taking into consideration the domain of the parameter change $\theta$, we get $-\tfrac{\pi}{2}\leqslant\tfrac{\pi}{2}\alpha\theta\leqslant\tfrac{\pi}{2}$. Thus, $\cos(\pi\alpha\theta/2)>0$ in the case $\alpha<1$ and $\theta\in[-1,1]$ or in the case $\alpha=1$ and $\theta\neq\pm1$. Using now this result in this formula, we obtain
\begin{equation}\label{eq:Ulim_ale1_rx2inf_phi=0}
  \lim_{\stackrel{R\to\infty}{x\to\infty}}V(0,R,x,\alpha,\theta)=0\quad\mbox{if}\quad \alpha<1\quad \mbox{and}\quad -1\leqslant\theta\leqslant1\quad\mbox{or}\quad \alpha=1\quad \mbox{and}\quad -1<\theta<1
\end{equation}

For the case $\varphi=\pi$ we have
\begin{equation*}
  \lim_{\stackrel{R\to\infty}{x\to\infty}}V(\pi,R,x,\alpha,\theta)=
  \lim_{\stackrel{R\to\infty}{x\to\infty}}\frac{R}{\pi x}\exp\left\{-(R/x)^\alpha\cos\left(\alpha(\pi+\tfrac{\pi}{2}\theta)\right)\right\}.
\end{equation*}
To have this limit equal to zero, it is necessary to have $\cos\left(\alpha(\pi+\tfrac{\pi}{2}\theta)\right)>0$. From here we get
$-1/\alpha-2<\theta<1/\alpha-2$. Considering that $\alpha\leqslant1$ and $-1\leqslant\theta\leqslant1$, it is clear that the obtained inequality is not satisfied for any admissible $\theta$. Consequently, the case $\varphi=\pi$ should be excluded from consideration.

We will return to the integral $I_{\Gamma_R}$. Using in the integral (\ref{eq:I_gamR_lim}) the expressions (\ref{eq:Ulim_a<1_rx2int}) and (\ref{eq:Ulim_ale1_rx2inf_phi=0}) we obtain
\begin{equation}\label{eq:I_gamR_lim_a<1}
  \lim_{\stackrel{R\to\infty}{x\to\infty}}I_{\Gamma_R}=0, \quad \mbox{if}\quad 0\leqslant\varphi_0\leqslant\pi-\epsilon,
\end{equation}
where $\epsilon$ is an arbitrary positive number such as $\epsilon\to0$. This result is valid in the case $0<\alpha<1$ and $-1\leqslant\theta\leqslant1$ or $\alpha=1$ and $-1<\theta<1$.

Now we consider the case $1<\alpha\leqslant2$. With such values of the parameter $\alpha$ the asymmetry parameter changes within the limits $-(\tfrac{2}{\alpha}-1)\leqslant\theta\leqslant(\tfrac{2}{\alpha}-1)$.  We have
\begin{equation*}
   \lim_{\stackrel{R\to\infty}{x\to\infty}} V(\varphi,R,x,\alpha,\theta)=
    \lim_{\stackrel{R\to\infty}{x\to\infty}} \frac{R}{\pi x}\exp\left\{-R\sin\varphi-(R/x)^\alpha\cos\left(\alpha\left(\varphi+\tfrac{\pi}{2}\theta\right)\right)\right\}
\end{equation*}
Since $\alpha>1$ we obtain that $(R/x)^\alpha$ increases quicker than $R$ at $R\to\infty$. Consequently, the summand $R\sin\varphi$ can be neglected. As a result
\begin{equation*}
\lim_{\stackrel{R\to\infty}{x\to\infty}}V(\varphi,R,x,\alpha,\theta)=\lim_{\stackrel{R\to\infty}{x\to\infty}}\frac{R}{\pi x}\exp\left\{-(R/x)^\alpha\cos\left(\alpha\left(\varphi+\tfrac{\pi}{2}\theta\right)\right)\right\}.
\end{equation*}
This limit will be equal to zero if $\cos\left(\alpha\left(\varphi+\tfrac{\pi}{2}\theta\right)\right)>0$ or $-\frac{\pi}{2}<\alpha\left(\varphi+\tfrac{\pi}{2}\theta\right)<\frac{\pi}{2}$. From here we obtain that
\begin{equation}\label{eq:Ulim_a>1_rx2inf_phi}
   \lim_{\stackrel{R\to\infty}{x\to\infty}}V(\varphi,R,x,\alpha,\theta)=0,\quad\mbox{if}\quad -\frac{\pi}{2\alpha}-\frac{\pi}{2}\theta<\varphi<\frac{\pi}{2\alpha}-\frac{\pi}{2}\theta.
\end{equation}

We consider the extreme points of this interval and exactly $\varphi=-\frac{\pi}{2\alpha}-\frac{\pi}{2}\theta$ and $\varphi=\frac{\pi}{2\alpha}-\frac{\pi}{2}\theta$. It is easy to see that at these points $\cos\left(\alpha\left(\varphi+\tfrac{\pi}{2}\theta\right)\right)=0$. In the case $\varphi=\frac{\pi}{2\alpha}-\frac{\pi}{2}\theta$ we have
\begin{equation}\label{eq:Ulim_a>1_rx2inf_phimax_tmp}
  \lim_{\stackrel{R\to\infty}{x\to\infty}} V\left(\frac{\pi}{2\alpha}-\frac{\pi}{2}\theta,R,x,\alpha,\theta\right)=
  \lim_{\stackrel{R\to\infty}{x\to\infty}}\frac{R}{\pi x}\exp\left\{-R\sin\left(\frac{\pi}{2\alpha}-\frac{\pi}{2}\theta\right)\right\}
\end{equation}
Now we will examine the domain of the admissible values of the sine argument. Substituting alternately in the expression $\frac{\pi}{2\alpha}-\frac{\pi}{2}\theta$ the values $\alpha=1$ and $\alpha=2$ we get
\begin{equation*}
\frac{\pi}{2\alpha}-\frac{\pi}{2}\theta= \frac{\pi}{2}-\frac{\pi}{2}\theta,\quad \mbox{if }\alpha=1,\quad
\qquad \frac{\pi}{2\alpha}-\frac{\pi}{2}\theta=-\frac{\pi}{4}+\frac{\pi}{2}\theta,\quad\mbox{if }\alpha=2.
\end{equation*}

Considering that $-(\frac{2}{\alpha}-1)\leqslant\theta\leqslant\frac{2}{\alpha}-1$ we obtain that in the case $\alpha=1$ the parameter $\theta$ takes the values from the range $-1\leqslant\theta\leqslant1$, and in the case $\alpha=2$ the asymmetry parameter takes only one possible value $\theta=0$. Thus,
\begin{equation*}
  \frac{\pi}{2}-\frac{\pi}{2}\theta=\begin{cases}
                                       \pi, & \mbox{if } \alpha=1,\ \theta=-1, \\
                                       0, & \mbox{if } \alpha=1,\ \theta=1,
                                     \end{cases} \qquad
  -\frac{\pi}{4}+\frac{\pi}{2}\theta=\frac{\pi}{4},\ \mbox{if } \alpha=2.
\end{equation*}
Consequently, for all other values of the parameters $\alpha,\theta$, lying in the ranges  $1<\alpha\leqslant2$ and   $-(\frac{2}{\alpha}-1)\leqslant\theta\leqslant\frac{2}{\alpha}-1$ the sine argument in the expression (\ref{eq:Ulim_a>1_rx2inf_phimax_tmp}) takes the values from the range  $0<\frac{\pi}{2\alpha}-\frac{\pi}{2}\theta<\pi$. From here it follows that
\begin{equation*}
  \sin\left(\frac{\pi}{2\alpha}-\frac{\pi}{2}\theta\right)>0, \quad\mbox{if}\quad 1<\alpha\leqslant2\quad \mbox{and}\quad   -\left(\frac{2}{\alpha}-1\right)\leqslant\theta\leqslant\frac{2}{\alpha}-1.
\end{equation*}

Now using this result in (\ref{eq:Ulim_a>1_rx2inf_phimax_tmp}), we obtain
\begin{equation}\label{eq:Ulim_a>1_rx2inf_phimax}
  \lim_{\stackrel{R\to\infty}{x\to\infty}} V\left(\frac{\pi}{2\alpha}-\frac{\pi}{2}\theta,R,x,\alpha,\theta\right)=0, \quad 1<\alpha\leqslant2\quad \text{and}\quad   -\left(\frac{2}{\alpha}-1\right)\leqslant\theta\leqslant\frac{2}{\alpha}-1.
\end{equation}
It should be pointed out that the case $\alpha=1$ is excluded from the result since it was examined in the previous case.

Now we consider the case when $\varphi=-\frac{\pi}{2\alpha}-\frac{\pi}{2}\theta$. We obtain
\begin{equation}\label{eq:Ulim_a>1_rx2inf_phimin_tmp}
  \lim_{\stackrel{R\to\infty}{x\to\infty}} V\left(-\frac{\pi}{2\alpha}-\frac{\pi}{2}\theta,R,x,\alpha,\theta\right)=
  \lim_{\stackrel{R\to\infty}{x\to\infty}}\frac{R}{\pi x}\exp\left\{R\sin\left(\frac{\pi}{2\alpha}+\frac{\pi}{2}\theta\right)\right\}
\end{equation}
Considering that $1<\alpha\leqslant2$ and $-(2/\alpha-1)\leqslant\theta\leqslant 2/\alpha-1$, we obtain that the sine argument in this expression takes the values from the range  $0<\frac{\pi}{2\alpha}+\frac{\pi}{2}\theta<\pi$. Thus,
\begin{equation*}
  \sin\left(\frac{\pi}{2\alpha}+\frac{\pi}{2}\theta\right)>0\quad\mbox{if}\quad 1<\alpha\leqslant2\quad \mbox{and}\quad   -\left(\frac{2}{\alpha}-1\right)\leqslant\theta\leqslant\frac{2}{\alpha}-1.
\end{equation*}
Now using this result in (\ref{eq:Ulim_a>1_rx2inf_phimin_tmp}), we obtain
\begin{equation}\label{eq:Ulim_a>1_rx2inf_phimin}
  \lim_{\stackrel{R\to\infty}{x\to\infty}} V\left(-\frac{\pi}{2\alpha}-\frac{\pi}{2}\theta,R,x,\alpha,\theta\right)=\infty,
  \quad\mbox{if}\quad 1<\alpha\leqslant2\quad \mbox{and}\quad   -\left(\frac{2}{\alpha}-1\right)\leqslant\theta\leqslant\frac{2}{\alpha}-1.
\end{equation}
Consequently, this case is necessary to exclude.

Combining now (\ref{eq:Ulim_a>1_rx2inf_phi}) and (\ref{eq:Ulim_a>1_rx2inf_phimax}) and taking into account (\ref{eq:Ulim_a>1_rx2inf_phimin}) we obtain that
\begin{equation*}
   \lim_{\stackrel{R\to\infty}{x\to\infty}} V\left(\varphi,R,x,\alpha,\theta\right)=0, \quad\mbox{if}\quad -\frac{\pi}{2\alpha}-\frac{\pi}{2}\theta<\varphi\leqslant\frac{\pi}{2\alpha}-\frac{\pi}{2}\theta.
\end{equation*}
Now we return to the integral $I_{\Gamma_R}$. Using in (\ref{eq:I_gamR_lim}) the previous expression we obtain
\begin{equation}\label{eq:I_gamR_lim_a>1}
  \lim_{\stackrel{R\to\infty}{x\to\infty}} I_{\Gamma_R}=0, \quad\mbox{if}\quad -\frac{\pi}{2\alpha}-\frac{\pi}{2}\theta+\epsilon\leqslant\varphi_0\leqslant\frac{\pi}{2\alpha}-\frac{\pi}{2}\theta,\quad R/x\to\infty,\quad 1<\alpha\leqslant2,
\end{equation}
where $\epsilon$ is the arbitrary positive real number in such a way that $\epsilon\to0$.

Combining now the expressions (\ref{eq:I_gamR_lim_a<1}) and (\ref{eq:I_gamR_lim_a>1}) we get that in the case when $R\to\infty$ and $x\to\infty$ and the relation $R/x\to\infty$ for the limit we get the following result
\begin{equation*}
  \lim_{\stackrel{R\to\infty}{x\to\infty}}I_{\Gamma_R}=0,\quad\mbox{if}\quad
  \begin{cases}
    0<\alpha<1, -1\leqslant\theta\leqslant1, & 0\leqslant\varphi_0\leqslant\pi-\epsilon \\
    \alpha=1, -1<\theta<1, & 0\leqslant\varphi_0\leqslant\pi-\epsilon,\\
    1<\alpha\leqslant2, |\theta|\leqslant2/\alpha-1,& -\frac{\pi}{2\alpha}-\frac{\pi}{2}\theta+\epsilon\leqslant\varphi_0\leqslant\frac{\pi}{2\alpha}-\frac{\pi}{2}\theta.
  \end{cases}
\end{equation*}
Thus, the first, the second and third items of the lemma have been proved.

\textit{The case $R\to\infty$, $x\to\infty$, $R/x\approx1$.} We consider the passage to the limit (\ref{eq:I_gamR_lim}) in the case when $R$ and $x$ go to infinity but their relation $R/x\approx1$. We still are interested in the conditions under which the integral in (\ref{eq:I_gamR_lim}) will be equal to zero. The behavior of this integral is determined by the behavior of the integrand $V(\varphi,R,x,\alpha,\theta)$. Since in the considered case the relation $R/x\approx1$, then from (\ref{eq:Vfun}) it is clear that in the exponent the summand with the multiplier $R/x$ can be neglected in comparison with the summand containing the multiplier $R$.  We obtain
\begin{multline}\label{eq:Ulim_rx_apr_1}
  \lim_{\stackrel{R\to\infty}{x\to\infty}} V(\varphi,R,x,\alpha,\theta)=
  \lim_{\stackrel{R\to\infty}{x\to\infty}} \frac{R}{\pi x}\exp\left\{-R\sin\varphi- (R/x)^\alpha\cos\left(\alpha\left(\varphi+\tfrac{\pi}{2}\theta\right)\right)\right\}\\
  =\lim_{\stackrel{R\to\infty}{x\to\infty}} \frac{R}{\pi x}\exp\left\{-R\sin\varphi\right\}=0, \quad 0<\varphi<\pi,\quad 0<\alpha\leqslant2.
\end{multline}
It is easy to see that in the case $\varphi=0$ and $\varphi=\pi$ the limit $ \lim_{\stackrel{R\to\infty}{x\to\infty}} V(\varphi,R,x,\alpha,\theta)\neq0$.
Therefore, these two items should be excluded from consideration. Now using (\ref{eq:Ulim_rx_apr_1}) в (\ref{eq:I_gamR_lim}) it is clear that in the considered case when $R\to\infty$ and $x\to\infty$ but $R/x\approx1$
\begin{equation*}
  \lim_{\stackrel{R\to\infty}{x\to\infty}} I_{\Gamma_R}=0, \quad 0+\epsilon\leqslant\varphi\leqslant\pi-\epsilon,\quad 0<\alpha\leqslant2,
\end{equation*}
where $\epsilon$ is an arbitrary positive real number in a way that $\epsilon\to0$. This expression proves item 4 of the lemma.

\textit{The case $R\to\infty$, $x\to\infty$, $R/x\to0$.} Now we consider the case when in the integral (\ref{eq:I_gamR_lim}) $R$ and $x$ go to infinity but their relation $R/x$ go to zero. As in the previous cases, the behavior of the integral in (\ref{eq:I_gamR_lim}) will be determined by the behavior of the integrand $V(\varphi,R,x,\alpha,\theta)$. Taking into account that $R/x\to0$, then in the exponent in (\ref{eq:Vfun}) one can neglect the summand with the multiplier ($R/x$) in comparison to the summand with the multiplier $R$. As a result, we get
\begin{equation}\label{eq:Ulim_rx20}
  \lim_{\stackrel{R\to\infty}{x\to\infty}} V(\varphi,R,x,\alpha,\theta)=
  \lim_{\stackrel{R\to\infty}{x\to\infty}} \frac{R}{\pi x}\exp\left\{-R\sin\varphi\right\}=0,\quad 0<\varphi<\pi,\quad 0<\alpha\leqslant2.
\end{equation}

We will consider the extreme points of the obtained range of the variable $\varphi$, i.e. $\varphi=0$ and $\varphi=\pi$. Taking into account that $R/x\to0$,  in the case $\varphi=0$ we obtain
\begin{equation}\label{eq:Ulim_rx20_phi=0}
  \lim_{\stackrel{R\to\infty}{x\to\infty}} V(0,R,x,\alpha,\theta) =\lim_{\stackrel{R\to\infty}{x\to\infty}} \frac{R}{\pi x} \exp\left\{-(R/x)^\alpha\cos\left(\tfrac{\pi}{2}\alpha\theta\right)\right\}=0, \quad 0<\alpha\leqslant2.
\end{equation}
In the same way, in the case $\varphi=\pi$ we get
\begin{equation}\label{eq:Ulim_rx20_phi=pi}
\lim_{\stackrel{R\to\infty}{x\to\infty}} V(\pi,R,x,\alpha,\theta)=0,\quad 0<\alpha\leqslant2.
\end{equation}

Now using the expressions (\ref{eq:Ulim_rx20}), (\ref{eq:Ulim_rx20_phi=0}), (\ref{eq:Ulim_rx20_phi=pi}) в (\ref{eq:I_gamR_lim}), we obtain
\begin{equation*}
  \lim_{\stackrel{R\to\infty}{x\to\infty}} I_{\Gamma_R} =0,\quad 0\leqslant\varphi\leqslant\pi,\quad 0<\alpha\leqslant2,\quad R/x\to0.
\end{equation*}
As a result, the fifth item of the lemma has been proved.

\textit{The case $R\to0$, $x\to\infty$.} It is left to consider the case $R\to0$ and $x\to\infty$. Preforming the limit passage in (\ref{eq:I_GammaR}) we get
\begin{equation}\label{eq:I_gamR_lim_Rto0}
  \lim_{\stackrel{R\to0}{x\to\infty}} I_{\Gamma_R}\leqslant  \lim_{\stackrel{R\to0}{x\to\infty}}\int_{\varphi_1}^{\varphi_2} V(\varphi,R,x,\alpha,\theta)d\varphi.
\end{equation}
As in other cases this limit is determined with the limit $\lim_{\stackrel{R\to0}{x\to\infty}} V(\varphi,R,x,\alpha,\theta)$. Now using the expression (\ref{eq:Vfun}) here it is easy to see that
\begin{equation*}
   \lim_{\stackrel{R\to0}{x\to\infty}} V(\varphi,R,x,\alpha,\theta)=0,\quad 0<\alpha\leqslant2,\quad -\pi\leqslant\varphi\leqslant\pi.
\end{equation*}
Now returning to (\ref{eq:I_gamR_lim_Rto0}) we obtain that $\varphi_1=-\pi$, $\varphi_2=\pi$ and
\begin{equation*}
   \lim_{\stackrel{R\to0}{x\to\infty}} I_{\Gamma_R}=0,\quad -\pi\leqslant\varphi\leqslant\pi,\quad 0<\alpha\leqslant2.
\end{equation*}
The obtained expression proves the lemma completely.
\begin{flushright}
  $\Box$
\end{flushright}
\end{proof}
\begin{lemma}\label{lem:LoopInt}
We will consider the contour $\Gamma_2$ in the complex plane $\zeta=\tau e^{i\varphi}$ which is the ray that comes out of the point $\zeta=0$ and extends to infinity. Then, if moving along the contour $\Gamma_2$ from the point $\zeta=0$ we go to infinity in a way that at $x\to\infty$ one condition is met
\begin{enumerate}
  \item $0\leqslant\varphi\leqslant\pi-\epsilon,\quad \frac{\tau}{x}\to\infty,\quad 0<\alpha<1,\quad -1\leqslant\theta\leqslant1$;
  \item $0\leqslant\varphi\leqslant\pi-\epsilon,\quad \frac{\tau}{x}\to\infty,\quad \alpha=1,\quad -1<\theta<1$;
  \item $-\frac{\pi}{2\alpha}-\frac{\pi}{2}\theta+\epsilon\leqslant\varphi\leqslant\frac{\pi}{2\alpha}-\frac{\pi}{2}\theta,\quad \frac{\tau}{x}\to\infty,\quad 1<\alpha\leqslant2$;
  \item $0+\epsilon\leqslant\varphi\leqslant\pi-\epsilon$,\quad $\frac{\tau}{x}\approx1$, \quad $0<\alpha\leqslant2$;
  \item $0\leqslant\varphi\leqslant\pi,\quad \frac{\tau}{x}\to0,\quad 0<\alpha\leqslant2$, \label{cond:tu/x2int_lemm}
\end{enumerate}
then the equality turns out to be valid
\begin{multline}\label{eq:loopInt_eq}
  \frac{1}{\pi x}\int_{0}^{\infty}\exp\left\{ i\tau- \left(\frac{\tau}{x}\right)^\alpha \exp\left\{i\tfrac{\pi}{2}\alpha\theta\right\}\right\}d\tau\\
  =\frac{1}{\pi x}\int_{0}^{\infty}\exp\left\{i\tau e^{i\varphi}-\left(\frac{\tau}{x}\right)^\alpha \exp\left\{i\alpha(\varphi+\tfrac{\pi}{2}\theta)\right\}+i\varphi\right\}d\tau.
\end{multline}
Here $\epsilon$ is the arbitrary positive number in such a way that $\epsilon\to0$.
\end{lemma}
\begin{proof}
We will consider the integral (\ref{eq:I_Gamma}), we will choose in the complex plane $\zeta=\tau e^{i\varphi}$ the auxiliary contour $\Gamma$, which is defined in the following way (see Fig.~\ref{fig:loopC})
\begin{equation*}
  \Gamma=\left\{\begin{array}{c}
             \Gamma_1  =\{\zeta: \varepsilon\leqslant\tau\leqslant R, \varphi=0\}, \\
             \Gamma_R = \{\zeta: \tau=R, 0\leqslant\varphi\leqslant\varphi_0\},\\
             \Gamma_2 = \{\zeta: R\geqslant\tau\geqslant \varepsilon, \varphi=\varphi_0\}, \\
             \Gamma_\varepsilon= \{\zeta: \tau=\varepsilon, \varphi_0\geqslant\varphi\geqslant0\},
           \end{array}\right.
\end{equation*}
where $\varepsilon>0$ and  $R>0$ are arbitrary real numbers in a way that $\varepsilon<R$.  The contour is traversed in the positive direction.
\begin{figure}
  \centering
  \includegraphics[width=0.3\textwidth]{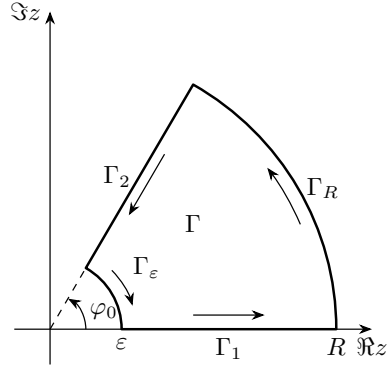}
  \caption{The contour of integration $\Gamma$}\label{fig:loopC}
\end{figure}
As a result, the integral $I(\Gamma)$ can be written in the form
\begin{multline*}
  I(\Gamma)=  \frac{1}{x\pi}\int_{\Gamma_1} e^{i\zeta}g^{+}(\zeta/x,\alpha,-\theta)d\zeta
  +\frac{1}{x\pi}\int_{\Gamma_R} e^{i\zeta}g^{+}(\zeta/x,\alpha,-\theta)d\zeta\\
  +\frac{1}{x\pi}\int_{\Gamma_2} e^{i\zeta}g^{+}(\zeta/x,\alpha,-\theta)d\zeta
  +\frac{1}{x\pi}\int_{\Gamma_\varepsilon} e^{i\zeta}g^{+}(\zeta/x,\alpha,-\theta)d\zeta=I(\Gamma_1)+I(\Gamma_R)+I(\Gamma_2)+I(\Gamma_\varepsilon).
\end{multline*}

Now we will let  $R$ go to infinity in this expression and $\varepsilon$ to zero. As a result, we get
\begin{equation}\label{eq:I_gam_lim}
  \lim_{\stackrel{R\to\infty,\varepsilon\to0}{x\to\infty}}I(\Gamma)
  =\lim_{\stackrel{R\to\infty,\varepsilon\to0}{x\to\infty}}I(\Gamma_1)
  +\lim_{\stackrel{R\to\infty}{x\to\infty}}I(\Gamma_R)
  +\lim_{\stackrel{R\to\infty,\varepsilon\to0}{x\to\infty}}I(\Gamma_2)
  +\lim_{\stackrel{\varepsilon\to0}{x\to\infty}}I(\Gamma_\varepsilon).
\end{equation}
We will suppose that for this integral condition 6 of lemma~\ref{lem:I_GammaR} is met and one of  conditions 1-5 of the same lemma. Then, in the framework of the assumptions made, lemma~\ref{lem:I_GammaR} gives that
\begin{equation}\label{eq:I_gam_eps_lim}
  \lim_{\stackrel{\varepsilon\to0}{x\to\infty}}I(\Gamma_\varepsilon)=0,
\end{equation}
since condition 6 of lemma ~\ref{lem:I_GammaR} is met and
\begin{equation}\label{eq:I_gam_R_lim}
  \lim_{\stackrel{R\to\infty}{x\to\infty}}I(\Gamma_R)=0,
\end{equation}
since one of conditions 1-5 of lemma~\ref{lem:I_GammaR} is met. Using (\ref{eq:I_Gamma}) and considering that $\zeta=\tau e^{i\varphi}$, for the integral $I(\Gamma_1)$ we obtain
\begin{multline}\label{eq:I_gam_1_lim}
  \lim_{\stackrel{R\to\infty,\varepsilon\to0}{x\to\infty}}I(\Gamma_1) = \lim_{\stackrel{R\to\infty,\varepsilon\to0}{x\to\infty}} \frac{1}{\pi x}\int_{\varepsilon}^{R}\exp\left\{i\tau -(\tau/x)^\alpha\exp\left\{i\tfrac{\pi}{2}\alpha\theta\right\}\right\}d\tau\\
  = \frac{1}{\pi x}\int_{0}^{\infty}\exp\left\{i\tau -(\tau/x)^\alpha\exp\left\{i\tfrac{\pi}{2}\alpha\theta\right\}\right\}d\tau.
\end{multline}
For the integral $I(\Gamma_2)$ we get
\begin{multline}\label{eq:I_gam_2_lim}
  \lim_{\stackrel{R\to\infty,\varepsilon\to0}{x\to\infty}}I(\Gamma_2) = \lim_{\stackrel{R\to\infty,\varepsilon\to0}{x\to\infty}}
  \frac{1}{\pi x}\int_{R}^{\varepsilon}\exp\left\{i\tau e^{i\varphi}-(\tau/x)^\alpha \exp\left\{i\alpha(\varphi+\tfrac{\pi}{2}\theta)\right\}+i\varphi\right\}d\tau\\
  =-\frac{1}{\pi x}\int_{0}^{\infty}\exp\left\{i\tau e^{i\varphi}-(\tau/x)^\alpha \exp\left\{i\alpha(\varphi+\tfrac{\pi}{2}\theta)\right\}+i\varphi\right\}d\tau.
\end{multline}

Next, we will take account of the fact that the integrand of the integral (\ref{eq:I_Gamma}) is an analytic function in the complex plane $\zeta$, and the contour $\Gamma$ is a closed contour. Hence,
\begin{equation*}
  \lim_{\stackrel{R\to\infty,\varepsilon\to0}{x\to\infty}}I(\Gamma)=0.
\end{equation*}
Using now this expression as well as the expressions (\ref{eq:I_gam_eps_lim}), (\ref{eq:I_gam_R_lim}), (\ref{eq:I_gam_1_lim}) and (\ref{eq:I_gam_2_lim}) in (\ref{eq:I_gam_lim}) we get
\begin{equation*}
  \frac{1}{\pi x}\int_{0}^{\infty}\exp\left\{ i\tau- \left(\frac{\tau}{x}\right)^\alpha \exp\left\{i\tfrac{\pi}{2}\alpha\theta\right\}\right\}d\tau=
  \frac{1}{\pi x}\int_{0}^{\infty}\exp\left\{i\tau e^{i\varphi}-\left(\frac{\tau}{x}\right)^\alpha \exp\left\{i\alpha(\varphi+\tfrac{\pi}{2}\theta)\right\}+i\varphi\right\}d\tau.
\end{equation*}
At the same time, conditions 1-5 of lemma~\ref{lem:I_GammaR} used for the proof (\ref{eq:I_gam_R_lim}) pass to conditions 1-5 of this lemma. Thus, the lemma is proved.
\begin{flushright}
  $\Box$
\end{flushright}
\end{proof}

As we can see, the proved lemma gives an opportunity to substitute the contour of integration in the integral $\frac{1}{\pi x}\int_{0}^{\infty}e^{i\tau} g(\tau/x,\alpha,-\theta)d\tau$ and  pass from integration along the positive part of the real axis to integration along the half line leaving the origin at the angle $\varphi$. This lemma turns out to be very useful in the next theorem, the proof of which we will now proceed to.

\begin{theorem}\label{theor:pdfExpansion}
In the case $x\to\pm\infty$ for any admissible set of parameters $(\alpha,\theta)$, with exception of the values $\theta=\pm1$, for the probability density $g(x,\alpha,\theta)$ the representation in the form of a power series is valid
\begin{equation}\label{eq:g_expan}
  g(x,\alpha,\theta)=g_N^{\infty}(|x|,\alpha,\theta^*)+R_N^{\infty}(|x|,\alpha,\theta^*),
\end{equation}
where $\theta^*=\theta\sign(x)$ and
\begin{equation}\label{eq:g_NInf}
  g_N^\infty(x,\alpha,\theta)=\frac{1}{\pi}\sum_{n=0}^{N-1}\frac{(-1)^{n+1}}{n!}\Gamma(\alpha n+1)\sin\left(\tfrac{\pi}{2}\alpha n(1+\theta)\right) x^{-\alpha n-1},\quad x>0,
\end{equation}
\begin{equation}\label{eq:R_NInf}
|R_N^{\infty}(x,\alpha,\theta)|\leqslant
  \frac{x^{-\alpha N-1}}{\pi N!}\left(\Gamma(\alpha N+1)+x^{-\alpha}\Gamma(\alpha(N+1)+1)\right),\quad x>0.
\end{equation}
\end{theorem}
\begin{proof}
  Without loss of generality, we assume that $x>0$. The case $x<0$ can be obtained with the help of  property~\ref{prop:Inversion} and formulas~(\ref{eq:InversionFormula}). We will perform the inverse Fourier transform of the characteristic function (\ref{eq:CF_formC}). To do this, we use the first relation in the inversion formula (\ref{eq:InverseFormula}). We obtain
  \begin{multline}\label{eq:pdfExpan_tmp0}
    g(x,\alpha,\theta)=\frac{1}{\pi}\Re \int_{0}^{\infty}e^{itx}\hat{g}(t,\alpha,-\theta)dt=
    \frac{1}{\pi}\Re\int_{0}^{\infty}\exp\left\{itx-t^\alpha \exp\left\{i\tfrac{\pi}{2}\alpha\theta\right\}\right\}dt\\
    = \frac{1}{\pi x}\Re\int_{0}^{\infty}\exp\left\{i\tau -(\tau/x)^\alpha\exp\left\{i\tfrac{\pi}{2}\alpha\theta\right\}\right\}d\tau,
  \end{multline}
  where, in the last equality, the integration variable $\tau=tx$ was substituted. Further, we will assume that $x\to\infty$. In addition, we will assume that $x$ will tend to infinity in a way that in the last integral at $\tau\to\infty$ the relation $\tau/x\to0$.  Then, if these conditions are met, the conditions of item~\ref{cond:tu/x2int_lemm} of Lemma~\ref{lem:LoopInt} turn to be met. Therefore, the last integral in the previous expression can be converted using the formula (\ref{eq:loopInt_eq}). As a result, we get
  \begin{equation}\label{eq:pdfExpan_tmp1}
    g(x,\alpha,\theta) =\frac{1}{\pi x}\Re\int_{0}^{\infty}\exp\left\{i\tau e^{i\varphi}-\left(\tfrac{\tau}{x}\right)^\alpha \exp\left\{i\alpha\left(\varphi+\tfrac{\pi}{2}\theta\right)\right\}+i\varphi\right\} d\tau.
  \end{equation}
  
  As it follows from item~\ref{cond:tu/x2int_lemm} of Lemma~\ref{lem:LoopInt} this formula is valid if the conditions $0\leqslant\varphi\leqslant\pi$, $\tau/x\to0$, $0<\alpha\leqslant2$ are met. Next, we take  account of the fact that the value of the argument $\varphi$ in this expression can take any arbitrary value from the interval $0\leqslant\varphi\leqslant\pi$. We choose it in the form
  \begin{equation}\label{eq:varphi0}
  \varphi=\frac{\pi}{2}-\frac{\pi}{2}\theta.
  \end{equation}
  As we can see, for any $-1\leqslant\theta\leqslant1$ inequalities $0\leqslant\frac{\pi}{2}-\frac{\pi}{2}\theta\leqslant\pi$ are satisfied. Now substituting the expression (\ref{eq:varphi0}) in the formula (\ref{eq:pdfExpan_tmp1}) we obtain
  \begin{multline*}
    g(x,\alpha,\theta)=\frac{1}{\pi x}\Re\int_{0}^{\infty} \exp\left\{-\tau e^{-i\frac{\pi}{2}\theta}-\left(\tfrac{\tau}{x}\right)^\alpha e^{i\alpha\frac{\pi}{2}}+i\left(\tfrac{\pi}{2}-\tfrac{\pi}{2}\theta\right)\right\}d\tau\\
    =\frac{1}{\pi x}\Re ie^{-i\frac{\pi}{2}\theta}\int_{0}^{\infty}\exp\left\{-\tau e^{-i\frac{\pi}{2}\theta}-\left(\tfrac{i\tau}{x}\right)^\alpha\right\}d\tau.
  \end{multline*}

  Considering that $x\to\infty$ and $\tau/x\to0$ we can expand $\exp\left\{-\left(\tfrac{i\tau}{x}\right)^\alpha\right\}$ into a Taylor series in the vicinity of the point $\tau/x=0$. As a result we get
  \begin{equation}\label{eq:g_expan_x>0}
    g(x,\alpha,\theta)=g_N^{\infty}(x,\alpha,\theta)+R_N^{\infty}(x,\alpha,\theta),\quad x>0,
  \end{equation}
  where $N$-th partial sum $g_N^\infty(x,\alpha,\theta)$ and  remainder term of the series $R_N^\infty(x,\alpha,\theta)$ have the form
  \begin{align}
    g_N^\infty(x,\alpha,\theta) & =\frac{1}{\pi x}\Re ie^{-i\frac{\pi}{2}\theta}\int_{0}^{\infty} \exp\left\{-\tau e^{-i\frac{\pi}{2}\theta}\right\}\sum_{n=0}^{N-1}\frac{(-1)^n}{n!}\left(\frac{i\tau}{x}\right)^{\alpha n}d\tau,\quad x>0,\nonumber\\
    R_N^\infty(x,\alpha,\theta) & = \frac{1}{\pi x}\Re ie^{-i\frac{\pi}{2}\theta}\int_{0}^{\infty} \exp\left\{-\tau e^{-i\frac{\pi}{2}\theta}\right\} R_N\left(-\left(\tfrac{i\tau}{x}\right)^\alpha\right) d\tau, \quad x>0.\label{eq:R_NInf_0}
  \end{align}
  Here $R_N(y)=\frac{y^N}{N!}e^{y\zeta}$, $(0<\zeta<1)$ is the remainder term in the Lagrange form.
  
 We consider $N$-th partial sum $g_N^\infty(x,\alpha,\theta)$. To calculate the integral, we will change the order of summation and integration
  \begin{equation}\label{eq:g_Ninf_1}
    g_N^\infty(x,\alpha,\theta)= \frac{1}{\pi x}\Re ie^{-i\frac{\pi}{2}\theta} \sum_{n=0}^{N-1}\frac{(-1)^n}{n!} \left(\frac{i}{x}\right)^{\alpha n} \int_{0}^{\infty} \exp\left\{-\tau e^{-i\frac{\pi}{2}\theta}\right\} \tau^{\alpha n}d\tau,\quad x>0.
  \end{equation}
 We will examine the range of the argument values $-\tfrac{\pi}{2}\theta$. The range of admissible values for the parameter $\theta$ is determined by the inequality $|\theta|\leqslant\min(1,2/\alpha-1)$. This shows if $\alpha\leqslant1$, then $-1\leqslant\theta\leqslant1$, and if $1<\alpha\leqslant2$, then $-(2/\alpha-1)\leqslant\theta\leqslant2/\alpha-1$. Thus, for any $0<\alpha\leqslant2$ we obtain $ -\tfrac{\pi}{2}\leqslant-\tfrac{\pi}{2}\theta\leqslant\tfrac{\pi}{2}$. It is important to note that the extreme values of this interval $\pm\tfrac{\pi}{2}$ are reached at $\alpha\leqslant1$ and $\theta=\mp1$.

 To calculate the integral in (\ref{eq:g_Ninf_1}) we will use one well-known formula given in\cite{Bateman_V1_1953} (see \S 1.5, formula~(31))
  \begin{equation*}
   \int_{0}^{\infty}t^{\gamma-1}e^{-ct\cos\beta-ict\sin\beta}dt=\Gamma(\gamma)c^{-\gamma}e^{-i\gamma\beta},\
   -\frac{\pi}{2}<\beta<\frac{\pi}{2},\ \Re\gamma>0\ \text{or}\ \beta=\pm\frac{\pi}{2},\ 0<\Re\gamma<1.
  \end{equation*}
  If we make use of the Euler formula $\cos\beta+i\sin\beta=e^{i\beta}$, then this integral can be represented in the form
  \begin{equation}\label{eq:Gamma_intRepr1}
   \int_{0}^{\infty}t^{\gamma-1}e^{-ct\exp\{i\beta\}}dt=\Gamma(\gamma)c^{-\gamma}e^{-i\gamma\beta},\quad
   -\frac{\pi}{2}<\beta<\frac{\pi}{2},\ \Re\gamma>0\ \text{or}\ \beta=\pm\frac{\pi}{2},\ 0<\Re\gamma<1.
  \end{equation}

  Comparing now the integral in (\ref{eq:Gamma_intRepr1}) and integral in (\ref{eq:g_Ninf_1}) we see that these two integrals coincide with the exception of the case $-\tfrac{\pi}{2}\theta=\pm\tfrac{\pi}{2}$. These two points do not enter the range of admissible values for the argument $\beta$ in the formula (\ref{eq:Gamma_intRepr1}). Consequently, if to exclude these two extreme values from consideration then  for the calculation of the integral in (\ref{eq:g_Ninf_1}) one can make use of (\ref{eq:Gamma_intRepr1}). Excluding from consideration the points $-\frac{\pi}{2}\theta=\pm\frac{\pi}{2}$, we obtain
   \begin{multline}\label{eq:g_NInf_2}
     g_N^\infty(x,\alpha,\theta)=\frac{1}{\pi x}\Re i e^{-i\tfrac{\pi}{2}\theta} \sum_{n=0}^{N-1}\frac{(-1)^n}{n!} \left(\frac{i}{x}\right)^{\alpha n} \Gamma(\alpha n+1) e^{i(\alpha n+1)\tfrac{\pi}{2}\theta}\\
     = \frac{1}{\pi}\sum_{n=0}^{N-1}\frac{(-1)^{n+1}}{n!}\Gamma(\alpha n+1) \sin\left(\tfrac{\pi}{2}\alpha n(1+\theta)\right) x^{-\alpha n-1},\quad x>0,
   \end{multline}
   where it was taken into account that $\Re i e^{-i\tfrac{\pi}{2}\theta} e^{i(\alpha n+1)\tfrac{\pi}{2}\theta}=-\sin\left(\tfrac{\pi}{2}\alpha n(1+\theta)\right)$.
   
   We consider now the remainder term $R_N^\infty(x,\alpha,\theta)$. From the expression  (\ref{eq:R_NInf_0}) we get
   \begin{equation*}
     R_N^\infty(x,\alpha,\theta)=\frac{(-1)^N}{\pi N!} x^{-\alpha N-1} \Re i^{\alpha N+1} e^{-i\tfrac{\pi}{2}\theta}\int_{0}^{\infty} \tau^{\alpha N} \exp\left\{-\tau \exp\left\{-i\tfrac{\pi}{2}\theta\right\} -\left(\tfrac{i\tau}{x}\right)^\alpha \zeta\right\} d\tau,\quad x>0.
   \end{equation*}
   Unfortunately, it is impossible to calculate the integral in this expression, since the exact value of $\zeta$ is unknown. We only know that $0<\zeta<1$. However, this integral can be estimated from above.
   
   Taking into account that we consider the case $x>0$, we obtain
   \begin{equation}\label{eq:absR_Ninf_1}
     R_N^\infty(x,\alpha,\theta)\leqslant |R_N^\infty(x,\alpha,\theta)| \leqslant
     \frac{x^{-\alpha N-1}}{\pi N!}\left|\Re i^{\alpha N+1} e^{-i\tfrac{\pi}{2}\theta}\int_{0}^{\infty}\tau^{\alpha N}e^{-\tau e^{-i\frac{\pi}{2}\theta}} \exp\left\{-\left(\tfrac{i\tau}{x}\right)^\alpha\zeta\right\} d\tau \right|
   \end{equation}
  We take into account that the case $\tau/x\to0$ is considered. As a result, the multiplier $\exp\left\{-\left(\tfrac{i\tau}{x}\right)^\alpha\zeta\right\}$ can be expanded into a Taylor series and leave only the summands of the first order of smallness in the obtained expansion. We have
   \begin{equation*}
   \exp\left\{-\left(\tfrac{i\tau}{x}\right)^\alpha\zeta\right\} =\sum_{k=0}^{\infty}\frac{(-1)^k}{k!}\left(\left(\frac{i\tau}{x}\right)^\alpha\zeta\right)^k\approx1-\zeta\left(\frac{i\tau}{x}\right)^\alpha
   \end{equation*}
  We also consider that $\left|i^{\alpha N+1} e^{-i\tfrac{\pi}{2}\theta}\right|\leqslant1$. As a result, the expression (\ref{eq:absR_Ninf_1}) takes the form
   \begin{multline}\label{eq:R_NInf_est_x>0}
     |R_N^\infty(x,\alpha,\theta)| \leqslant  \frac{x^{-\alpha N-1}}{\pi N!}\left|\Re  \int_{0}^{\infty}\tau^{\alpha N}\exp\left\{-\tau e^{-i\tfrac{\pi}{2}\theta}\right\} \left(1-\zeta\left(\frac{i\tau}{x}\right)^\alpha\right)d\tau\right|\\
     = \frac{x^{-\alpha N-1}}{\pi N!}\left|\Re  \left(\int_{0}^{\infty}\tau^{\alpha N}\exp\left\{-\tau e^{-i\tfrac{\pi}{2}\theta}\right\}d\tau - \zeta\left(\frac{i}{x}\right)^\alpha \int_{0}^{\infty}\tau^{\alpha (N+1)} \exp\left\{-\tau e^{-i\tfrac{\pi}{2}\theta}\right\}d\tau\right)\right|\\
     =\frac{x^{-\alpha N-1}}{\pi N!}\left|\Re\left(\Gamma(\alpha N+1) e^{i\tfrac{\pi}{2}\theta(\alpha N+1)}- \zeta x^{-\alpha}\Gamma(\alpha(N+1)+1)e^{i\tfrac{\pi}{2}\alpha +i\tfrac{\pi}{2}\theta(\alpha (N+1)+1)}\right)\right|\\
     \leqslant \frac{x^{-\alpha N-1}}{\pi N!}\left(\Gamma(\alpha N+1)+\zeta x^{-\alpha}\Gamma(\alpha(N+1)+1)\right)\\
     \leqslant \frac{x^{-\alpha N-1}}{\pi N!}\left(\Gamma(\alpha N+1)+x^{-\alpha}\Gamma(\alpha(N+1)+1)\right),\quad x>0.
   \end{multline}
   Here to calculate the integrals in the second quality the formula (\ref{eq:Gamma_intRepr1}) was used. In the passage in the penultimate inequality, it was taken into account that  $i^\alpha=e^{i\tfrac{\pi}{2}\alpha}$, $\left|e^{i\tfrac{\pi}{2}\theta(\alpha N+1)}\right|\leqslant1$ and $\left|e^{i\tfrac{\pi}{2}\alpha +i\tfrac{\pi}{2}\theta(\alpha (N+1)+1)}\right|\leqslant1$. In the last inequality it was assumed that $\zeta=1$.
   
   Thus, in the case $x>0$ the representation (\ref{eq:g_expan_x>0}) is valid for the probability density in which $g_N^\infty(x,\alpha,\theta)$  is determined by the power series (\ref{eq:g_NInf_2}), and the estimate (\ref{eq:R_NInf_est_x>0}) is valid for the remainder term $R_N^\infty(x,\alpha,\theta)$.  To obtain the expression for the case of negative $x$ we will make use of the inversion property (\ref{eq:InversionFormula}) for the probability density. From this formula, it is clear that the case $x<0$ is reduced to the case $x>0$ with the help of the inversion sign in the asymmetry parameter $\theta$. As we can see it is possible to combine the cases $x>0$ and $x<0$, if to introduce the parameter $\theta^*=\theta\sign x$ and take the coordinate $x$ in absolute value. As a result the formula (\ref{eq:g_expan_x>0}) takes the form (\ref{eq:g_expan}) which is valid both for positive and negative $x$. The obtained expression proves the theorem completely.
   \begin{flushright}
     $\Box$
   \end{flushright}
\end{proof}

The proved theorem gives a representation for the probability density of a strictly stable law in the form of a power series in the cases $x\to\pm\infty$. However, depending on the value of the characteristic exponent $\alpha$ this power series is either convergent or asymptotic. This property of the expansion of the probability density in a power series was pointed out in the introduction. We will examine in more detail the question of the convergence of the resulting power series and formulate the result as a corollary to the theorem which has just been proved.

\begin{corollary}\label{corol:pdfCoverg}
    In the case $\alpha<1$ the series (\ref{eq:g_NInf}) converges for all $x$. In this case for the probability density $g(x,\alpha,\theta)$ for all $\theta$, satisfying the condition $-1<\theta<1$, the expansion in the form of an infinite series is valid
    \begin{equation*}
    g(x,\alpha,\theta)=\frac{1}{\pi}\sum_{n=0}^{\infty}\frac{(-1)^{n+1}}{n!}\Gamma(\alpha n+1)\sin\left(\tfrac{\pi}{2}\alpha n (1+\theta^*)\right)|x|^{-\alpha n-1}.
    \end{equation*}
    In the case $\alpha=1$ the series  (\ref{eq:g_NInf}) converges for all $x$, satisfying the condition $|x|>1$. In this case the probability density $g(x,1,\theta)$ for all $\theta\neq\pm1$ can be represented in the form of an infinite series
    \begin{equation}\label{eq:g_NInf_a1}
      g(x,1,\theta)=\frac{1}{\pi}\sum_{n=0}^{\infty}(-1)^{n+1}\sin\left(\tfrac{\pi}{2}n(1+\theta^*)\right)|x|^{-n-1}, \quad |x|>1
    \end{equation}
    In the case $\alpha>1$ the series (\ref{eq:g_NInf}) is convergent at $N\to\infty$. In this case for the probability density $g(x,\alpha,\theta)$ for any admissible $\theta$ the representation in the form of an asymptotic series is valid
    \begin{equation*}
    g(x,\alpha,\theta)\sim\frac{1}{\pi}\sum_{n=0}^{N-1}\frac{(-1)^{n+1}}{n!}\Gamma(\alpha n+1)\sin\left(\tfrac{\pi}{2}\alpha n (1+\theta^*)\right)|x|^{-\alpha n-1},\quad x\to\pm\infty.
    \end{equation*}
\end{corollary}

\begin{proof}
    As in the previous theorem, we will consider the case $x>0$ and examine the convergence of the series (\ref{eq:g_NInf}). As one can see, this series is sign-alternating. Therefore, the following estimates are valid
    \begin{multline*}
      g_N^\infty(x,\alpha,\theta)\leqslant\left|g_N^\infty(x,\alpha,\theta)\right|\leqslant \frac{1}{\pi}\sum_{n=0}^{N-1}\frac{\Gamma(\alpha n+1)}{\Gamma(n+1)}\left|(-1)^{n+1} \sin(\tfrac{\pi}{2}\alpha n(1+\theta))\right| x^{-\alpha n-1}\\
      \leqslant \frac{1}{\pi}\sum_{n=0}^{N-1}\frac{\Gamma(\alpha n+1)}{\Gamma(n+1)} x^{-\alpha n-1},\quad x>0.
    \end{multline*}
    
    We will make use of the Cauchy criterion in the limiting form and the Stirling formula
    \begin{equation}\label{eq:Stirling}
      \Gamma(z)\sim e^{-z}z^{z-\frac{1}{2}}\sqrt{2\pi},\quad z\to\infty,\quad |\arg z|<\pi.
    \end{equation}
    As a result, we get
    \begin{multline}\label{eq:RN_term}
      \lim_{n\to\infty} \left(\frac{\Gamma(\alpha n+1)}{\pi \Gamma(n+1)} x^{-\alpha n-1}\right)^{\frac{1}{n}}
      =\lim_{n\to\infty} \frac{x^{-\alpha-\frac{1}{n}}}{\pi^{1/n}} \left(\frac{\exp\left\{-\alpha n-1\right\} (\alpha n+1)^{\alpha n+1-1/2}\sqrt{2\pi}}{\exp\left\{-n-1\right\} (n+1)^{n+1-1/2}\sqrt{2\pi}}\right)^{\frac{1}{n}} \\
      =\lim_{n\to\infty}\frac{x^{-\alpha-\frac{1}{n}}}{\pi^{\frac{1}{n}}} \frac{e^{-\alpha-\frac{1}{n}} (\alpha n+1)^{\alpha+\frac{1}{2n}}}{e^{-1-\frac{1}{n}}(n+1)^{1+\frac{1}{2n}}}
      = \lim_{n\to\infty}\frac{x^{-\alpha-\frac{1}{n}}}{\pi^{1/n}} e^{-\alpha+1}(\alpha n)^{\alpha+
      \frac{1}{2n}} n^{-1-\frac{1}{2n}}\\
      = x^{-\alpha} e^{1-\alpha} \alpha^\alpha\lim_{n\to\infty}(x\pi)^{-\frac{1}{n}}\alpha^{\frac{1}{2n}}n^{\alpha-1}
      =x^{-\alpha} e^{1-\alpha} \alpha^\alpha\lim_{n\to\infty}n^{\alpha-1}
      =\begin{cases}
         0, & \mbox{if } \alpha<1, \\
         x^{-1}, & \mbox{if } \alpha=1,\\
         \infty, & \mbox{if } \alpha>1.
       \end{cases}
    \end{multline}
    Here, when passing in the third equality, the assumption was used that $\alpha n+1\approx \alpha n$ and $n+1\approx n$ at $n\to\infty$.   It can be seen from the obtained expression that in the case of $\alpha<1$ the series (\ref{eq:g_NInf}) converges for all $x>0$. In the case $\alpha>1$ this series diverges for all $x>0$. In the case $\alpha=1$ the series (\ref{eq:g_NInf}) converges for all $x>1$. At $0<x\leqslant1,$ and $\alpha=1$ this series will diverge. We will consider these three cases separately.
    
    We will examine the remainder term (\ref{eq:R_NInf}). Taking into account that we consider the case $x>0$, from the expression (\ref{eq:g_expan}) we obtain
    \begin{multline}\label{eq:g-g_NInf}
      |g(x,\alpha,\theta)-g_N^\infty(|x|,\alpha,\theta^*)|=\\
       = |g(x,\alpha,\theta)-g_N^\infty(x,\alpha,\theta)|\leqslant \frac{x^{-\alpha N-1}}{\pi N!}\left(\Gamma(\alpha N+1)+x^{-\alpha}\Gamma(\alpha(N+1)+1)\right).
    \end{multline}
    We will choose some arbitrary $x$ and fix it. We will calculate the limit of the right side of this expression at $N\to\infty$
    \begin{multline}\label{eq:lim_RN}
      \frac{1}{\pi}\lim_{N\to\infty}\frac{x^{-\alpha N-1}}{\Gamma(N+1)}\left(\Gamma(\alpha N+1)+x^{-\alpha}\Gamma(\alpha(N+1)+1)\right) \\
      =\frac{1}{\pi}\lim_{N\to\infty}x^{-\alpha N-1}\left(\frac{e^{-\alpha N-1} (\alpha N+1)^{\alpha N+1/2} \sqrt{2\pi}+ x^{-\alpha}e^{-\alpha(N+1)-1}(\alpha(N+1)+1)^{\alpha(N+1)+1/2}\sqrt{2\pi}}{e^{-N-1}(N+1)^{\alpha(N+1)+1/2}\sqrt{2\pi}}\right)\\
      = \frac{1}{\pi}\lim_{N\to\infty}\left(e^{N(1-\alpha)} (\alpha N+1)^{\alpha N+\frac{1}{2}}(N+1)^{-N-\frac{1}{2}}+x^{-\alpha} e^{-\alpha(N+1)+N}(\alpha(N+1)+1)^{\alpha(N+1)+\frac{1}{2}}(N+1)^{-N-\frac{1}{2}}\right)\\
      =\frac{1}{\pi}\lim_{N\to\infty}x^{-\alpha N-1} e^{N(1-\alpha)}\alpha^{\alpha N+\frac{1}{2}} N^{N(\alpha-1)}\left(1+|x|^{-\alpha}\right)\\
      =\frac{1}{\pi}\lim_{N\to\infty} \exp\left\{\alpha N\ln\alpha -N(1-\alpha)(\ln N-1)\right\} \left(1+x^{-\alpha}\right)x^{-\alpha N-1}=
      \begin{cases}
        0, & \mbox{if } \alpha<1 \\
        0, & \mbox{if } \alpha=1,\ x>1 \\
        \infty, & \mbox{if } \alpha=1,\ 0<x<1 \\
        \infty, & \mbox{if } \alpha>1.
      \end{cases}
    \end{multline}
    Here the Stirling formula (\ref{eq:Stirling}) was used and when passing in the third equality, it was taken into account that $N+1\approx N$  and $\alpha N+1\approx \alpha N$ at $N\to\infty$.
    
    One can see from the obtained expression that in the case of $x>0$ and $\alpha<1$ the right side of the expression (\ref{eq:g-g_NInf}) is an element of an infinitesimal sequence. This means that the sequence $g_N^\infty(x,\alpha,\theta)$ at $x>0$ converges to the probability density $g(x,\alpha,\theta)$ at $N\to\infty$. Therefore, at any fixed $x>0$ the probability density can be represented in the form an infinite series
    \begin{equation*}
     g(x,\alpha,\theta)=\frac{1}{\pi}\sum_{n=0}^{\infty}\frac{(-1)^{n+1}}{n!}\Gamma(\alpha n+1)\sin\left(\tfrac{\pi}{2}\alpha n (1+\theta)\right)x^{-\alpha n-1},\quad \alpha<1.
    \end{equation*}
    
    According to property~\ref{prop:Inversion} the case $x<0$ is reduced to the case $x>0$ by inverting the sign of the asymmetry parameter $\theta$. If we introduce the parameter $\theta^*=\theta\sign(x)$ and take the coordinate $x$ in absolute value, then one can combine these two cases. As a result, the previous expression will take the form
    \begin{equation*}
      g(x,\alpha,\theta)=\frac{1}{\pi}\sum_{n=0}^{\infty}\frac{(-1)^{n+1}}{n!}\Gamma(\alpha n+1)\sin\left(\tfrac{\pi}{2}\alpha n (1+\theta^*)\right)|x|^{-\alpha n-1},\quad \alpha<1.
    \end{equation*}
    which is valid for any $x$.  This proves the first part of the corollary.
    
   Now we will consider the case $\alpha=1$. It follows from the expression (\ref{eq:lim_RN}) that at the value $x>1$ the right side of the expression (\ref{eq:g-g_NInf}) is an element of an infinitesimal sequence. Taking also into account that in this case the series (\ref{eq:g_NInf}) is convergent, we obtain that the sequence $g_N^\infty(x,1,\theta)$ converges to the density $g(x,1,\theta)$ at $N\to\infty$. Consequently, in the case $\alpha=1$  for any fixed $x>1$ the representation in the form of an infinite series is valid for the probability density
    \begin{equation*}
      g(x,1,\theta)=\frac{1}{\pi}\sum_{n=0}^{\infty}(-1)^{n+1}\sin\left(\tfrac{\pi}{2}n(1+\theta)\right)x^{-n-1},\quad x>1.
    \end{equation*}
    
    Since the case $x>0$ is considered then, this formula is valid only for positive $x$. To generalize this formula to the case of negative $x$ is possible if we use the property of inversion (\ref{eq:InversionFormula}). In the same way as we did in the previous case, we will introduce the parameter $\theta^*=\theta\sign(x)$ and consider $|x|$. This gives an opportunity to combine the cases of negative and positive $x$. As a result, the previous formula takes the form
    \begin{equation*}
      g(x,1,\theta)=\frac{1}{\pi}\sum_{n=0}^{\infty}(-1)^{n+1}\sin\left(\tfrac{\pi}{2}n(1+\theta^*)\right)|x|^{-n-1},\quad |x|>1,
    \end{equation*}
    which is valid for both positive and negative $x$.  Thus, the second item of the corollary is proved.
    
    In the case $\alpha>1$ from the expression (\ref{eq:RN_term}) and (\ref{eq:lim_RN}) follows that in this case the series (\ref{eq:g_NInf}) diverges at $N\to\infty$. However, from the expression (\ref{eq:R_NInf}) follows that for some fixed  $N$
    \begin{equation*}
      R_N^\infty(x,\alpha,\theta)=O\left(x^{-\alpha N-1}\right),\quad x\to\infty.
    \end{equation*}
    Thus, for each fixed $N$ at $x>0$ from the expression (\ref{eq:g_expan}) we obtain
    \begin{equation*}
      g(x,\alpha,\theta)=\frac{1}{\pi}\sum_{n=0}^{N-1}\frac{(-1)^{n+1}}{n!}\Gamma(\alpha n+1)\sin\left(\tfrac{\pi}{2}\alpha n (1+\theta)\right)x^{-\alpha n-1} +O\left(x^{-\alpha N-1}\right), \quad x\to\infty.
    \end{equation*}
    
    This result can be generalized for the case of negative $x$ with the help of the inversion property. Doing in the same way as in the previous cases we get
    \begin{equation*}
      g(x,\alpha,\theta)=\frac{1}{\pi}\sum_{n=0}^{N-1}\frac{(-1)^{n+1}}{n!}\Gamma(\alpha n+1)\sin\left(\tfrac{\pi}{2}\alpha n (1+\theta^*)\right)|x|^{-\alpha n-1} +O\left(|x|^{-\alpha N-1}\right), \quad x\to\pm\infty.
    \end{equation*}
    As a result, we obtained the definition of the asymptotic series. Consequently
    \begin{equation*}
      g(x,\alpha,\theta)\sim\frac{1}{\pi}\sum_{n=0}^{N-1}\frac{(-1)^{n+1}}{n!}\Gamma(\alpha n+1)\sin\left(\tfrac{\pi}{2}\alpha n (1+\theta^*)\right)|x|^{-\alpha n-1}, \quad x\to\pm\infty, \quad \alpha>1.
    \end{equation*}
    Thus, the corollary has been proved completely.
    \begin{flushright}
      $\Box$
    \end{flushright}
\end{proof}

\begin{remark}\label{remark:pdf_a1}
In the case $\alpha=1$ for any $-1<\theta<1$ in the domain $|x|>1$ the series (\ref{eq:g_NInf_a1}) converges to the density (\ref{eq:g(x)_a=1}).
\end{remark}
\begin{proof}
  For the proof we will consider the density (\ref{eq:g(x)_a=1}) and will show that the expansion of this density into a Taylor series at $x\to\infty$ has the form (\ref{eq:g_NInf_a1}). Using the reduction formulas $\cos\left(\tfrac{\pi}{2}\theta\right)=\sin\left(\tfrac{\pi}{2}+\tfrac{\pi}{2}\theta\right)$, $\sin\left(\tfrac{\pi}{2}\theta\right)=-\cos\left(\tfrac{\pi}{2}+\tfrac{\pi}{2}\theta\right)$ we write the densities (\ref{eq:g(x)_a=1}) in the form
  \begin{equation}\label{eq:g(x)_a=1_tmp0}
    g(x,1,\theta)=\frac{\sin\left(\frac{\pi}{2}(1+\theta)\right)}{\pi\left(x^2+2x\cos\left(\frac{\pi}{2}(1+\theta)\right)+1\right)}.
  \end{equation}
  
  Next, since we need to find the expansion of the density in a series at $x\to\infty$ we will substitute the variable $y=1/x$ in this expression. With substitution of the variable like this the behavior of the density $g(x,1,\theta)$ at $x\to\infty$ will correspond to the behavior of the density $g(y,1,\theta)$ at $y\to0$. Thus, if we change the variable $x=1/y$ in the expression (\ref{eq:g(x)_a=1_tmp0})  and expand the obtained expression into a Taylor series in the vicinity of the point $y=0$, and then, in the obtained expression we again go to the variable $x$, then we will get the expansion of the density  (\ref{eq:g(x)_a=1}) at $x\to\infty$.
  
  We consider the expression $g(x,1,\theta)dx$, which represents the measure element $dP$. Performing the substitution of the variable $x=1/y$ we obtain
    \begin{equation*}
      g(x,1,\theta)dx=- \frac{\sin\left(\frac{\pi}{2}(1+\theta)\right)}{\pi\left(y^{-2}+2y^{-1}\cos\left(\frac{\pi}{2}(1+\theta)\right)+1\right)} \frac{dy}{y^2}=  g(y,1,\theta)dy,
    \end{equation*}
    where
    \begin{equation*}
      g(y,1,\theta)=-\frac{\sin\left(\frac{\pi}{2}(1+\theta)\right)}{\pi\left(y^{2}+2y\cos\left(\frac{\pi}{2}(1+\theta)\right)+1\right)}
    \end{equation*}
  
  Now we represent this function in the following form
    \begin{equation*}
      g(y,1,\theta)=-\frac{\sin\left(\frac{\pi}{2}(1+\theta)\right)}{\pi} f(g(y)),
    \end{equation*}
    where
    \begin{equation}\label{eq:f_g_def}
      f\equiv f(g)=1/g,\quad g\equiv g(y)=y^{2}+2y\cos\left(\tfrac{\pi}{2}(1+\theta)\right)+1.
    \end{equation}
   
   Now we will expand this function $g(y,1,\theta)$ into a Taylor series in the vicinity of the point $y=0$. Since this function is infinitely differentiable, we get
    \begin{multline}\label{eq:g(y)_seres}
      g(y,1,\theta)=g(0,1,\theta)+ \sum_{n=1}^{\infty}\frac{1}{n!}\left.\frac{d^n g(y,1,\theta)}{dy^n}\right|_{y=0}y^n\\
      =-\frac{1}{\pi}\sin\left(\tfrac{\pi}{2}(1+\theta)\right)-
      \frac{1}{\pi}\sin\left(\tfrac{\pi}{2}(1+\theta)\right)\sum_{n=1}^{\infty}\frac{1}{n!}\left.\frac{d^n f(g(y))}{dy^n}\right|_{y=0}y^n.
    \end{multline}
   
   Taking into account that the function $f(g(y))$ is a complex function then to calculate the $n$-th derivative from this function we will use Bruno’s formula (see~\cite{Riordan1958_en})
    \begin{equation}\label{eq:df(g(y))}
     \frac{d^n f(g(y))}{dy^n}=Y_n(fg_1,fg_2,\dots,fg_n),
    \end{equation}
    where $Y_n(fg_1,fg_2,\dots,fg_n)$ are Bell polynomials
    \begin{equation}\label{eq:BellPolinom}
     Y_n(fg_1,fg_2,\dots,fg_n)=\sum \frac{n!f_m}{k_1!k_2!\dots k_n!}\left(\frac{g_1}{1!}\right)^{k_1} \left(\frac{g_2}{2!}\right)^{k_2}\dots \left(\frac{g_n}{n!}\right)^{k_n}.
    \end{equation}
    Here $f^m\equiv f_m$, $m=k_1+k_2+\dots+k_n$, the sum is taken for all solutions to the equation
    \begin{equation}\label{eq:BelPolSumEq}
     k_1+2k_2+\dots+n k_n=n,\quad k_j\geqslant0,\quad j=1,2,\dots,n.
    \end{equation}
    and
    \begin{equation*}
     f_m=\left.\frac{d^mf(g)}{dg^m}\right|_{g=g(y)},\quad g_j=\frac{d^j g(y)}{dy^j }.
    \end{equation*}
    
    In view of (\ref{eq:f_g_def}), we obtain
    \begin{equation}\label{eq:fk_g}
     f_m=(-1)^m\left.\frac{m!}{g^{m+1}}\right|_{g=g(y)}=\frac{(-1)^m m!}{\left(y^2+2y \cos\left(\tfrac{\pi}{2}(1+\theta)\right)+1\right)^{m+1}}.
    \end{equation}
    For coefficients $g_j$ we have
    \begin{equation}\label{eq:gk_g}
     g_1=2y+2\cos\left(\tfrac{\pi}{2}(1+\theta)\right),\quad g_2=2,\quad g_3 =g_4=\dots=g_n=0.
    \end{equation}
    This shows that in the expression (\ref{eq:BellPolinom}) the summands remain in the sum that satisfy the equation
    \begin{equation}\label{eq:BelPolSumEq_g}
     k_1+2k_2=n.
    \end{equation}  
  Indeed, in the expression (\ref{eq:BellPolinom}) the summation is performed over all solutions to the equation (\ref{eq:BelPolSumEq}). In case, if the solution $k_j\neq0,\ j=3,4,\dots,n$, then the corresponding summand in the sum will be equal to zero, since $g_j=0$, $j=3,4,\dots,n$. If $k_j=0$, $j=3,4,\dots,n$, then the multiplier $(g_j/j!)^{k_j}=1$, since  $0^0=1$.  Therefore, in the expression (\ref{eq:BellPolinom}) there are terms in the sum that satisfy the solution to the equation (\ref{eq:BelPolSumEq_g}). This greatly simplifies the summation. From the equation (\ref{eq:BelPolSumEq_g}) it follows that $k_1=n-2k_2$. Considering that $k_1\geqslant0$ and $k_2\geqslant0$, we get $k_2=0,1,2,\dots,\left[\tfrac{n}{2}\right]$, where $[A]$ denotes the integer part of the number $A$. This gives us an opportunity in the sum (\ref{eq:BellPolinom})  to enter the summation index directly.
  
  In view of the foregoing, the formula (\ref{eq:BellPolinom}) becomes
    \begin{equation*}
     Y_n(fg_1,fg_2)=\sum_{k=0}^{\left[\frac{n}{2}\right]}\frac{n!f_{n-k}}{(n-2k)!k!}\left(\frac{g_1}{1!}\right)^{n-2k} \left(\frac{g_2}{2!}\right)^k,
    \end{equation*}
    where the relation $k_1=n-2k_2$ was used and the summation index $k\equiv k_2$ was introduced. Now substituting this relation in (\ref{eq:df(g(y))}) and using (\ref{eq:fk_g}) and (\ref{eq:gk_g}), we obtain
    \begin{equation*}
     \frac{d^n f(g(y))}{dy^n}=\sum_{k=0}^{\left[\tfrac{n}{2}\right]} \frac{(-1)^{n-k} n! (n-k)!}{k!(n-2k)!}\frac{\left(2y+2\cos\left(\tfrac{\pi}{2}(1+\theta)\right)\right)^{n-2k}} {\left(y^2+2y\cos\left(\tfrac{\pi}{2}(1+\theta)\right)+1\right)^{n-k+1}}.
    \end{equation*}
    
    Now we will calculate the value of this derivative at the point $y=0$. It is easy to see that
    \begin{equation}\label{eq:dg(0)dx}
     \left.\frac{d^n f(g(y)}{dy^n}\right|_{y=0}=(-1)^n n!\sum_{k=0}^{\left[\frac{n}{2}\right]} \frac{(-1)^{k} (n-k!)}{k!(n-2k)!}\left(2\cos\left(\tfrac{\pi}{2}(1+\theta)\right)\right)^{n-2k},
    \end{equation}
    where it was taken into account that $(-1)^{-k}=(-1)^k$.

    Next, we will use the general formula for $\sin(n\varphi)$ (see, for example, \cite{Schaum2018})
    \begin{equation*}
     \sin(n\varphi)=\sin\varphi\sum_{k=0}^{\left[\frac{n-1}{2}\right]}(-1)^{k}\frac{(n-k-1)!}{k!(n-2k-1)!}(2\cos\varphi)^{n-2k-1}.
    \end{equation*}
    Using this formula in (\ref{eq:dg(0)dx}), we get
    \begin{equation*}
    \left.\frac{d^n f(g(y))}{dy^n}\right|_{x=0}=\frac{(-1)^n n!}{\sin\left(\tfrac{\pi}{2}(1+\theta)\right)}\sin\left(\tfrac{\pi}{2}(n+1)(1-\theta)\right).
    \end{equation*}
    
    Now using this expression in (\ref{eq:g(y)_seres}), we obtain
    \begin{multline*}
     g(y,1,\theta)=-\frac{1}{\pi}\sin\left(\tfrac{\pi}{2}(1+\theta)\right)-
     \sum_{n=1}^{\infty}\frac{(-1)^n}{\pi}\sin(\tfrac{\pi}{2}(n+1)(1+\theta)) y^n\\
     = -\sum_{n=0}^{\infty}\frac{(-1)^n}{\pi}\sin(\tfrac{\pi}{2}(n+1)(1+\theta)) y^n.
    \end{multline*}
    We will get back to the variable $x$. By substituting the variable $x=1/y$, we obtain
    \begin{equation*}
      g(y,1,\theta)dy=\frac{dx}{x^2} \sum_{n=0}^{\infty}\frac{(-1)^n}{\pi}\sin(\tfrac{\pi}{2}(n+1)(1+\theta)) x^{-n}= g(x,1,\theta)dx,
    \end{equation*}
    where
    \begin{equation*}
      g(x,1,\theta)=\sum_{n=0}^{\infty}\frac{(-1)^n}{\pi}\sin(\tfrac{\pi}{2}(n+1)(1+\theta)) x^{-n-2}.
    \end{equation*}

    Substituting now in this expression the summation index $n+1=k$ and considering that $(-1)^{k-1}=(-1)^{k+1}$, we get
    \begin{equation*}
      g(x,1,\theta)=\sum_{k=1}^{\infty}\frac{(-1)^{k+1}}{\pi}\sin(\tfrac{\pi}{2}k(1+\theta)) x^{-k-1}
      =\sum_{k=0}^{\infty}\frac{(-1)^{k+1}}{\pi}\sin(\tfrac{\pi}{2}k(1+\theta)) x^{-k-1}.
    \end{equation*}
    Thus, the expansion of the plane (\ref{eq:g(x)_a=1}) into an infinite Taylor series at $x\to\infty$ exactly coincides with the series (\ref{eq:g_NInf_a1}). The domain of convergence of this series was examined in corollary~\ref{corol:pdfCoverg}, which shows that this series converges at $|x|>1$. This proves the corollary completely.
    \begin{flushright}
      $\Box$
    \end{flushright}
\end{proof}

Theorem~\ref{theor:pdfExpansion} makes it possible to calculate the probability density of a strictly stable law using the power series obtained in this theorem. For the practical implementation of this possibility, a criterion is needed that can help to determine the number of terms in the sum (\ref{eq:g_NInf}), required to calculate the density at a specified point $x$ with a specified accuracy. Such a criterion can be obtained if to use the estimate of the remainder term (\ref{eq:R_NInf}).

Indeed, from the expressions (\ref{eq:g_expan}) and (\ref{eq:R_NInf}) it follows that
\begin{equation*}	
  |g(x,\alpha,\theta)-g_N^{\infty}(|x|,\alpha,\theta^*)|\leqslant\frac{|x|^{-\alpha N-1}}{\pi N!} \left(\Gamma(\alpha N+1)+|x|^{-\alpha}\Gamma(\alpha(N+1)+1)\right)
\end{equation*}
If now for the given value $N$ we specify the value of the absolute error $$
|g(x,\alpha,\theta)-g_N^{\infty}(|x|,\alpha,\theta^*)|\leqslant \varepsilon, $$
then this allows us to introduce the threshold coordinate $x_\varepsilon^N$. The value of the threshold coordinate is found from the solution to the equation
\begin{equation}\label{eq:x_eps_eq}
  \varepsilon=\frac{\left|x_\varepsilon^N\right|^{-\alpha N-1}}{\pi N!} \left(\Gamma(\alpha N+1)+\left|x_\varepsilon^N\right|^{-\alpha}\Gamma(\alpha(N+1)+1)\right).
\end{equation}

The threshold coordinate determines the range of coordinates within which the absolute value of the calculation error using the series (\ref{eq:g_NInf}) will not exceed the specified level of accuracy $\varepsilon$. In the case under consideration, this domain is determined by the inequality $|x|\geqslant x_\varepsilon^N$. Thus,
\begin{equation*}
  \left|g(x,\alpha,\theta)-g_N^\infty(|x|,\alpha,\theta^*)\right|\leqslant\varepsilon,\quad |x|\geqslant x_\varepsilon^N.
\end{equation*}
Unfortunately, to obtain an explicit solution to the equation (\ref{eq:x_eps_eq}) with respect to the unknown $x_\varepsilon^N$ is not possible. Therefore, to solve this equation, one must use numerical methods that easily give an opportunity to find the threshold coordinate $x_\varepsilon^N$ for specified values $\varepsilon, N, \alpha$.

Taking into account the foregoing, we can write a formula for calculating the probability density for large $x$
\begin{equation}\label{eq:g_NInf_calc}
  g(x,\alpha,\theta)=g_N^\infty(|x|,\alpha,\theta^*),\quad |x|\geqslant x_\varepsilon^N.
\end{equation}
Here $g_N^\infty(x,\alpha,\theta)$ is determined by the series (\ref{eq:g_NInf}), and the threshold coordinate $x_\varepsilon^N$ is found from the solution to the equation (\ref{eq:x_eps_eq}) at specified values of the parameter $\alpha$, the number of summands $N$ in the sum (\ref{eq:g_NInf}) and level of accuracy $\varepsilon$. At the same time, it can be guaranteed that the absolute error in calculating the probability density using this formula will not exceed $\varepsilon$, but in reality it will be much less than this value.

\begin{figure}
  \centering
  \includegraphics[width=0.475\textwidth]{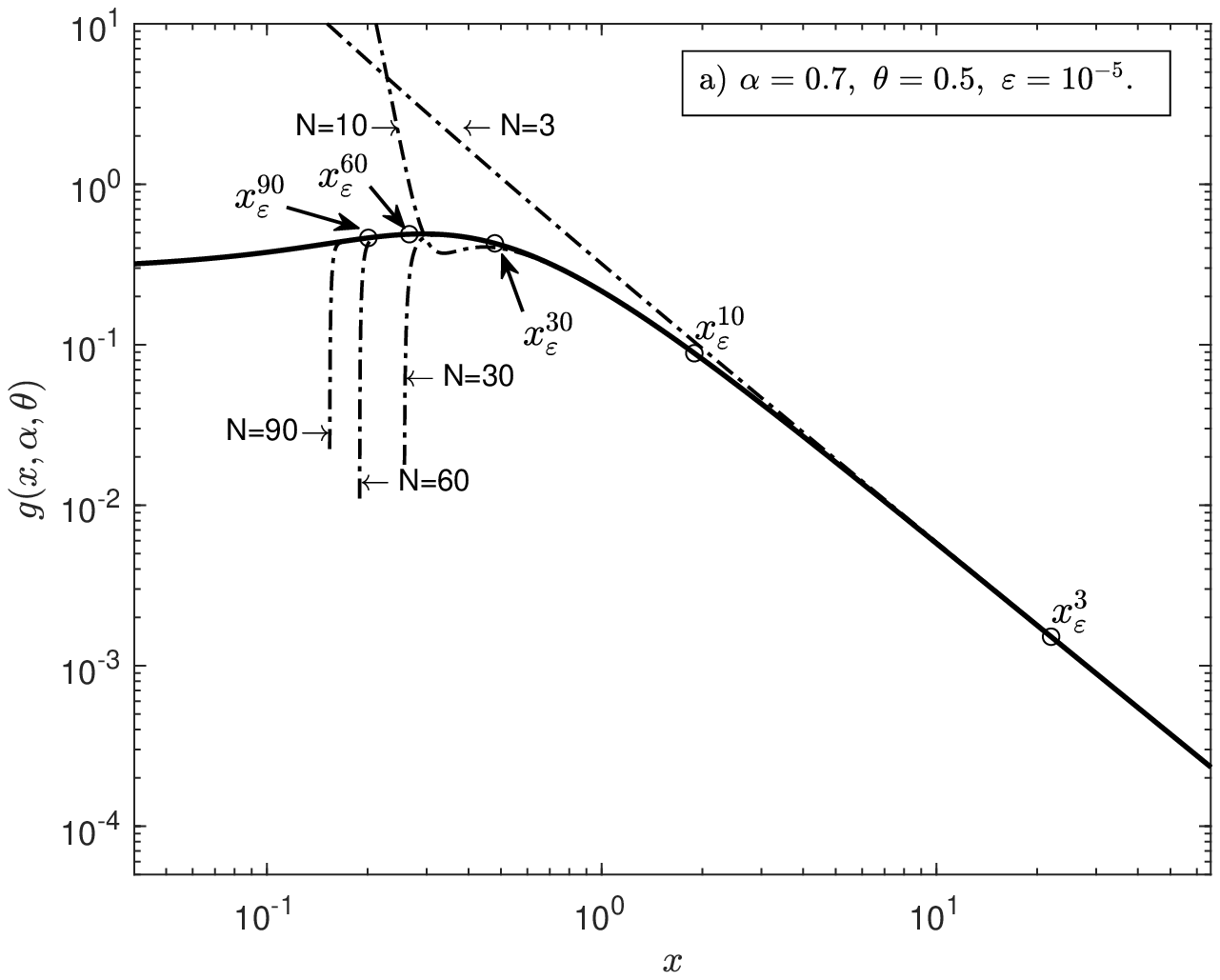}\hfill
  \includegraphics[width=0.475\textwidth]{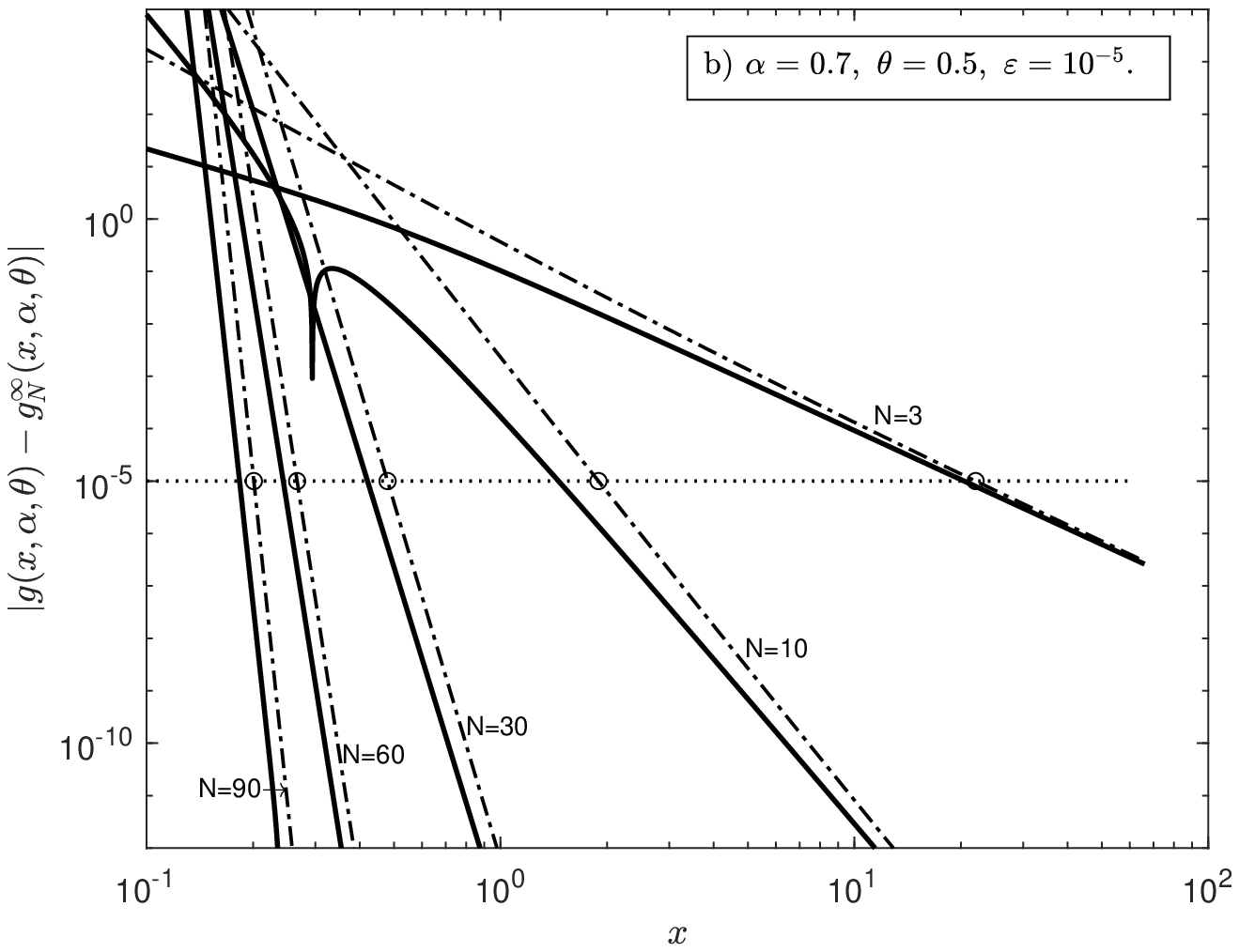}\\
  \caption{a) Probability density $g(x,\alpha,\theta)$ for the parameter values shown in the figure. The solid curve is the integral representation (\ref{eq:g(x)_int}), the dash-dotted curves are the formula  (\ref{eq:g_NInf_calc}) for different values of the number of terms $N$ in the sum (\ref{eq:g_NInf}). The circles show the position of the threshold coordinate $x_\varepsilon^N$ for corresponding values of $N$ and the specified accuracy level $\varepsilon$. (b) Graph of the absolute error of calculating the density $g(x,\alpha,\theta)$ using the formula (\ref{eq:g_NInf_calc}). The solid curves are the exact value of the absolute error $|g(x,\alpha,\theta)-g_N^\infty(x,\alpha,\theta)|$, the dash-dotted curves are the estimate of the remainder term (\ref{eq:R_NInf}), the dotted line is the specified accuracy level $\varepsilon$, the circles show the position of the threshold coordinate $x_\varepsilon^N$.}\label{fig:fig_a07}
\end{figure}
\begin{figure}
  \centering
  \includegraphics[width=0.475\textwidth]{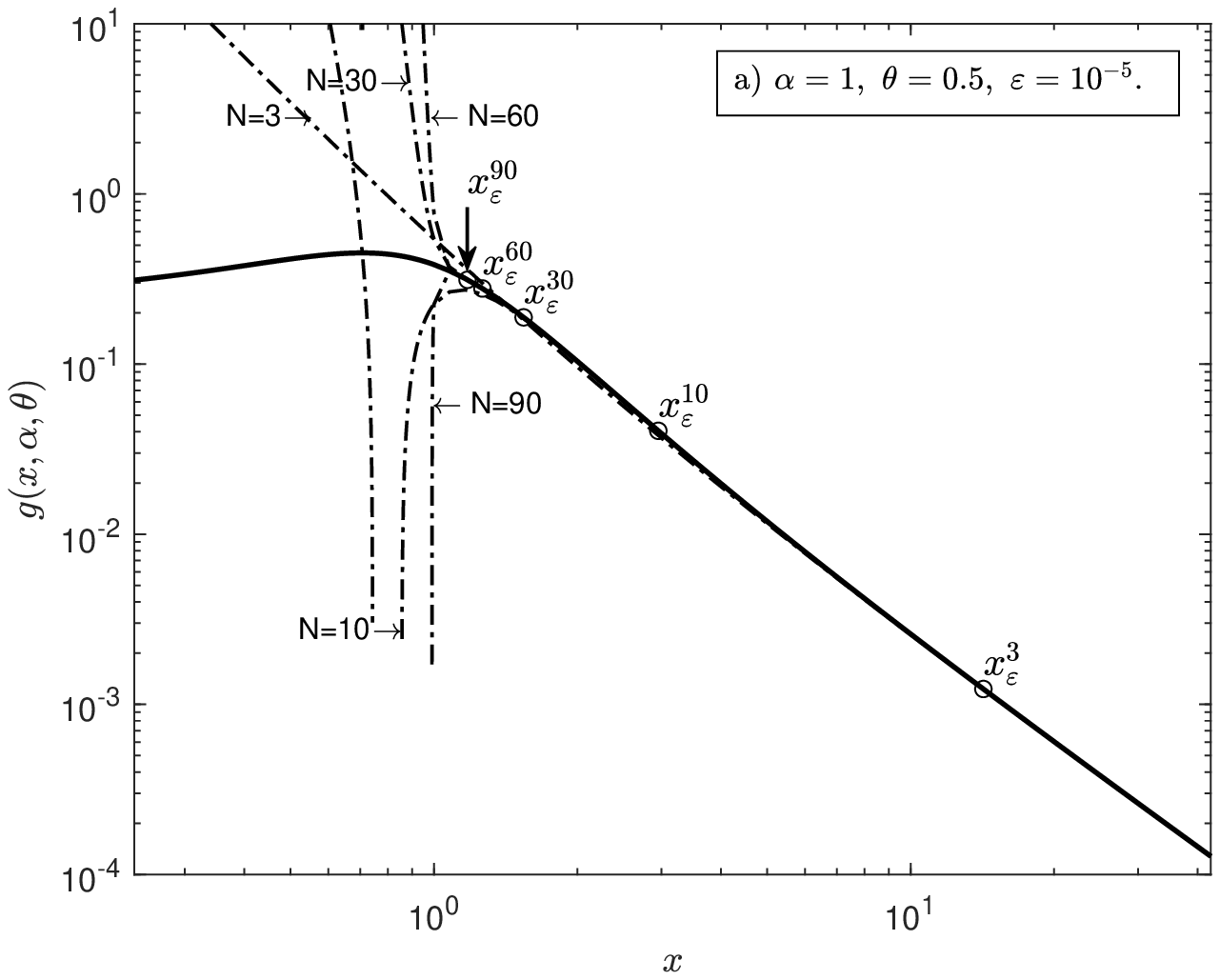}\hfill
  \includegraphics[width=0.475\textwidth]{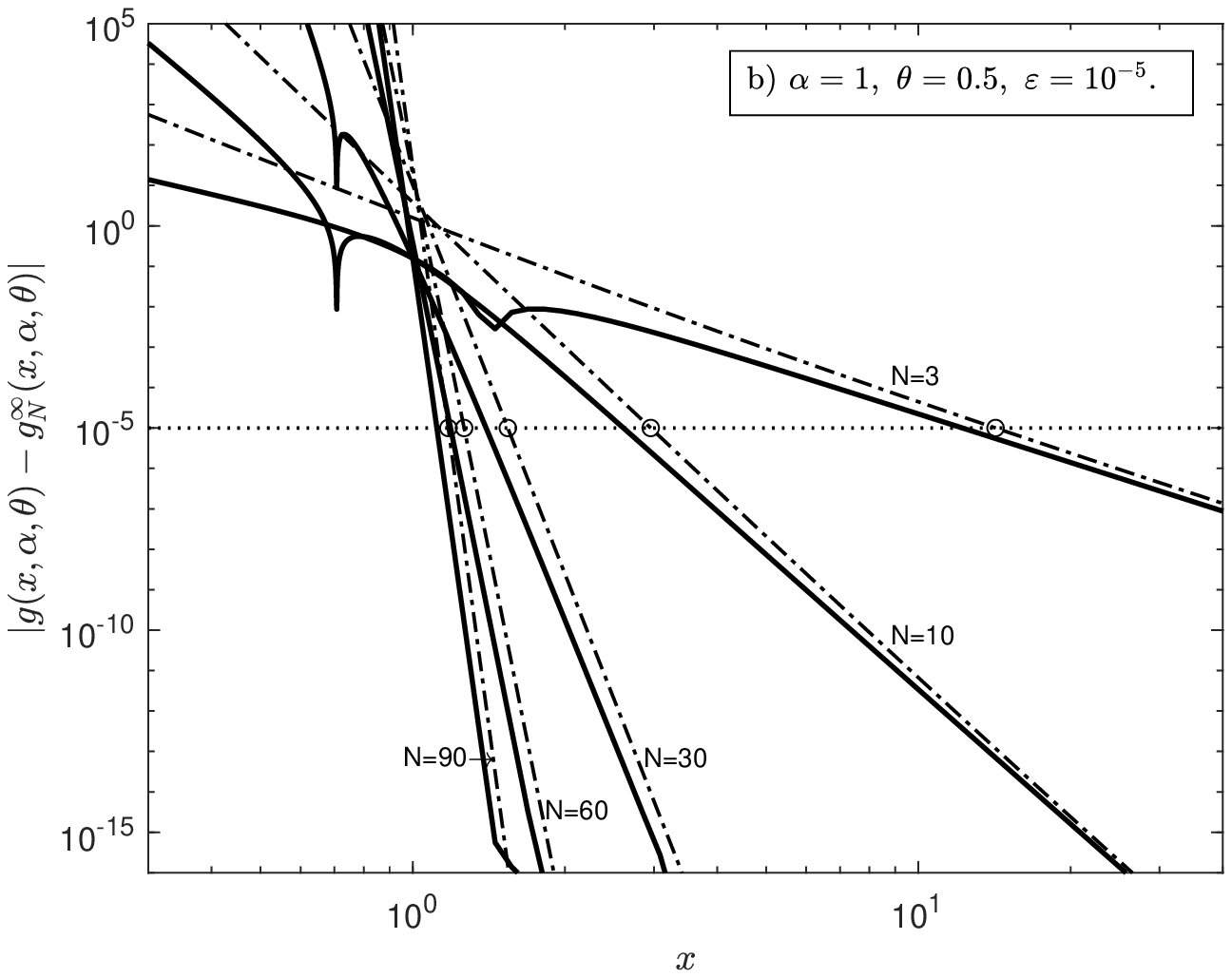}\\
  \caption{a) Probability density $g(x,\alpha,\theta)$ for the parameter values shown in the figure. The solid curve is the formula (\ref{eq:g(x)_a=1}), the dash-dotted curves are the formula (\ref{eq:g_NInf_calc}) for different values of the number of terms $N$ in the sum(\ref{eq:g_NInf}). The circles show the position of the threshold coordinate $x_\varepsilon^N$ for corresponding values of $N$ and the specified accuracy level $\varepsilon$. (b) ) Graph of the absolute error of calculating the density $g(x,\alpha,\theta)$ using the formula (\ref{eq:g_NInf_calc}). Solid curves are the exact value of the absolute error $|g(x,\alpha,\theta)-g_N^\infty(x,\alpha,\theta)|$, the dash-dotted curves are the estimate of the remainder term  (\ref{eq:R_NInf}), the dotted line is the specified accuracy level  $\varepsilon$, the circles show the position of the threshold coordinate   $x_\varepsilon^N$.}\label{fig:fig_a1}
\end{figure}

The results of calculating the probability density $g(x,\alpha,\theta)$ are given in Fig.~\ref{fig:fig_a07}a,~\ref{fig:fig_a1}a~and~\ref{fig:fig_a13}a for the specified values $\alpha$ and $\theta$ in the figures. In these figures the solid curve corresponds to the exact value of the probability density $g(x,\alpha,\theta)$. In the case $\alpha\neq1$  the integral representation (\ref{eq:g(x)_int}), was used for calculating the density $g(x,\alpha,\theta)$ and in the case $\alpha=1$ the formula  (\ref{eq:g(x)_a=1}) was used. The dashed-dotted curves correspond to the probability density calculated using the formula (\ref{eq:g_NInf_calc}) for values $N=3,10,30,60,90$. The circles denote the position of the threshold coordinate $x_\varepsilon^N$ for the selected level of accuracy $\varepsilon=10^{-5}$ and used values of $N$. Numerical values of the threshold coordinate were obtained using the numerical solution to the equation (\ref{eq:x_eps_eq}).   Figures~\ref{fig:fig_a07}b,~\ref{fig:fig_a1}b~and~\ref{fig:fig_a13}b demonstrate the absolute error of calculating the probability density using the formula (\ref{eq:g_NInf_calc}). In these figures the solid curves correspond to the exact value of the absolute error. To calculate it the formula  $\left|g(x,\alpha,\theta)-g_N^\infty(x,\alpha,\theta)\right|$ was used, where $g(x,\alpha,\theta)$ is the integral representation (\ref{eq:g(x)_int}) at $\alpha\neq1$ and the formula (\ref{eq:g(x)_a=1}) at $\alpha=1$, and $g_N^\infty(x,\alpha,\theta)$ is the formula (\ref{eq:g_NInf_calc}). The dashed-dotted curves correspond to the estimate of the remainder term (\ref{eq:R_NInf}), and the dotted straight line shows the position of the selected accuracy level $\varepsilon$. As in the previous figure, the circles show the position of the threshold coordinate $x_\varepsilon^N$ for the used values of $N$.

Figures~\ref{fig:fig_a07}b,~\ref{fig:fig_a1}b~and~\ref{fig:fig_a13}b clearly demonstrate that at the values $x\geqslant x_\varepsilon^N$ for all used values of $N$ the solid curves lie lower than the dashed-dotted curves. This means that the exact value of the absolute error is less than the estimate (\ref{eq:R_NInf}). Thus, at the values $|x|\geqslant x_\varepsilon^N$ the formula (\ref{eq:g_NInf_calc}) can be used for calculating the probability density. At the same time, it can be guaranteed that the value of the absolute error will not exceed the selected accuracy level $\varepsilon$, but in reality, it will be less than $\varepsilon$.

We will note one more feature, which is clearly visible from the presented figures. Figure~\ref{fig:fig_a07} shows calculations for the case $\alpha=0.7$. Figures~\ref{fig:fig_a07}a~and~\ref{fig:fig_a07}b show that as  $N$ increases the value of the threshold coordinate $x_\varepsilon^N$ decreases.  This is clearly seen from the position of the circles, each of which corresponds to the point $x_\varepsilon^N$ for each of the used $N$. It is quite clear that $x_\varepsilon^3>x_\varepsilon^{10}>x_\varepsilon^{30}>x_\varepsilon^{60}>x_\varepsilon^{90}$.  This behavior of the threshold coordinate (reducing the value of $x_\varepsilon^N$ with $N$ increasing) is the result of corollary~\ref{corol:pdfCoverg}. Indeed, this corollary shows that in the case of $\alpha<1$ at $N\to\infty$ the series (\ref{eq:g_NInf}) is convergent. Consequently, as the number of terms $N$ in the sum (\ref{eq:g_NInf}) increases, the accuracy of calculating the density $g(x,\alpha,\theta)$ at some fixed point $x$ will increase. This leads to the fact that at $\alpha<1$ the range of values of the coordinate $x$, at which the inequality $|g(x,\alpha,\theta)-g_N^\infty(x,\alpha,\theta)|\leqslant\varepsilon$ is valid, with an increase of $N$ expands. As a result, the value of the threshold coordinate $x_\varepsilon^N$, which is the solution to the equation $|g(x_\varepsilon^N,\alpha,\theta)-g_N^\infty(x_\varepsilon^N,\alpha,\theta)|=\varepsilon$, will decrease in absolute value with an increase of $N$, i.e. $\left|x_\varepsilon^N\right|\to0$ at $N\to\infty$.

This result can be obtained if in the equation (\ref{eq:x_eps_eq}) to pass to the limit $N\to\infty$. By passing to the limit and taking into account that $N+1\approx N$ at $N\to\infty$, for the left part of the equation (\ref{eq:x_eps_eq}) we obtain
\begin{equation*}
  \frac{1}{\pi\Gamma(N+1)}\left(\frac{\Gamma(\alpha N+1)}{|x_\varepsilon^N|^{\alpha N+1}}+ \frac{\Gamma(\alpha(N+1)+1)}{|x_\varepsilon^N|^{\alpha(N+1)+1}}\right)
  \approx  \frac{2}{\pi\Gamma(N+1)}\frac{\Gamma(\alpha N+1)}{|x_\varepsilon^N|^{\alpha N+1}}.
\end{equation*}
If now we substitute this expression instead of the left part of the equation (\ref{eq:x_eps_eq}), then we will see that the obtained equation is easy to solve relative to $|x_\varepsilon^N|$. As a result the solution to this equation has the form
\begin{equation*}
  \left|x_\varepsilon^N\right|=\left(\frac{2}{\pi\varepsilon}\frac{\Gamma(\alpha N+1)}{\Gamma(N)}\right)^{\frac{1}{\alpha N+1}}.
\end{equation*}

We find the limit of this solution at $N\to\infty$. Considering that $\alpha N+1\approx \alpha N$ at $N\to\infty$ and using the Stirling’s formula (\ref{eq:Stirling}),  we obtain
\begin{multline}\label{eq:x_eps_pdf_lim}
  \lim_{N\to\infty} \left|x_\varepsilon^N\right|
  =\lim_{N\to\infty} \left(\frac{2}{\pi\varepsilon}\frac{\Gamma(\alpha N+1)}{\Gamma(N)}\right)^{\frac{1}{\alpha N+1}}
  \approx\lim_{N\to\infty} \left(\frac{2}{\pi\varepsilon}\frac{\Gamma(\alpha N)}{\Gamma(N)}\right)^{\frac{1}{\alpha N}}\\
  =\lim_{N\to\infty} \left(\frac{2}{\pi\varepsilon} \frac{e^{-\alpha N}(\alpha N)^{\alpha N-1/2}\sqrt{2\pi}}{e^{-N} N^{N-1/2}\sqrt{2\pi}}\right)^{\frac{1}{\alpha N}}
  =\lim_{N\to\infty} \left(\frac{2}{\pi\varepsilon} e^{N(1-\alpha)}\alpha^{\alpha N-1/2}N^{N(\alpha-1)}\right)^{\frac{1}{\alpha N}}\\
  =e^{\frac{1-\alpha}{\alpha}}\alpha \lim_{N\to\infty}\left(\frac{2}{\pi\varepsilon\sqrt{\alpha}}\right)^{\frac{1}{\alpha N}}N^{\frac{\alpha-1}{\alpha}}=\begin{cases}
                                   0, & \alpha<1 \\
                                   1, & \alpha=1\\
                                   \infty, & \alpha>1.
                                 \end{cases}
\end{multline}

It is clear from the obtained expression that $\lim_{N\to\infty}\left|x_\varepsilon^N\right|=0$ if $\alpha<1$ . This means that at $\alpha<1$ predetermined calculation accuracy $\varepsilon$ is achieved on the entire number line $-\infty<x<\infty$. This result is a direct consequence of item~1 of corollary~\ref{corol:pdfCoverg}, where in the case $\alpha<1$ the series convergence (\ref{eq:g_NInf}) is proved on the entire number line at  $N\to\infty$.

A similar behavior of the threshold coordinate value is also observed in the case of $\alpha=1$ (see Fig.~\ref{fig:fig_a1}). The figure clearly shows that  $x_\varepsilon^3>x_\varepsilon^{10}>x_\varepsilon^{30}>x_\varepsilon^{60}>x_\varepsilon^{90}$, i.е. as the value of $N$ increases the value of the threshold coordinate $x_\varepsilon^N$ decreases. However, it is clear from the expression (\ref{eq:x_eps_pdf_lim}) that in this case $\lim_{n\to\infty}\left|x_\varepsilon^N\right|=1$. This means that in the case $\alpha=1$ at $N\to\infty$ the prespecified accuracy is achieved only in the range of coordinates $|x|>1$. This result is also a direct consequence of item 2 of corollary~\ref{corol:pdfCoverg}.
\begin{figure}
  \centering
  \includegraphics[width=0.475\textwidth]{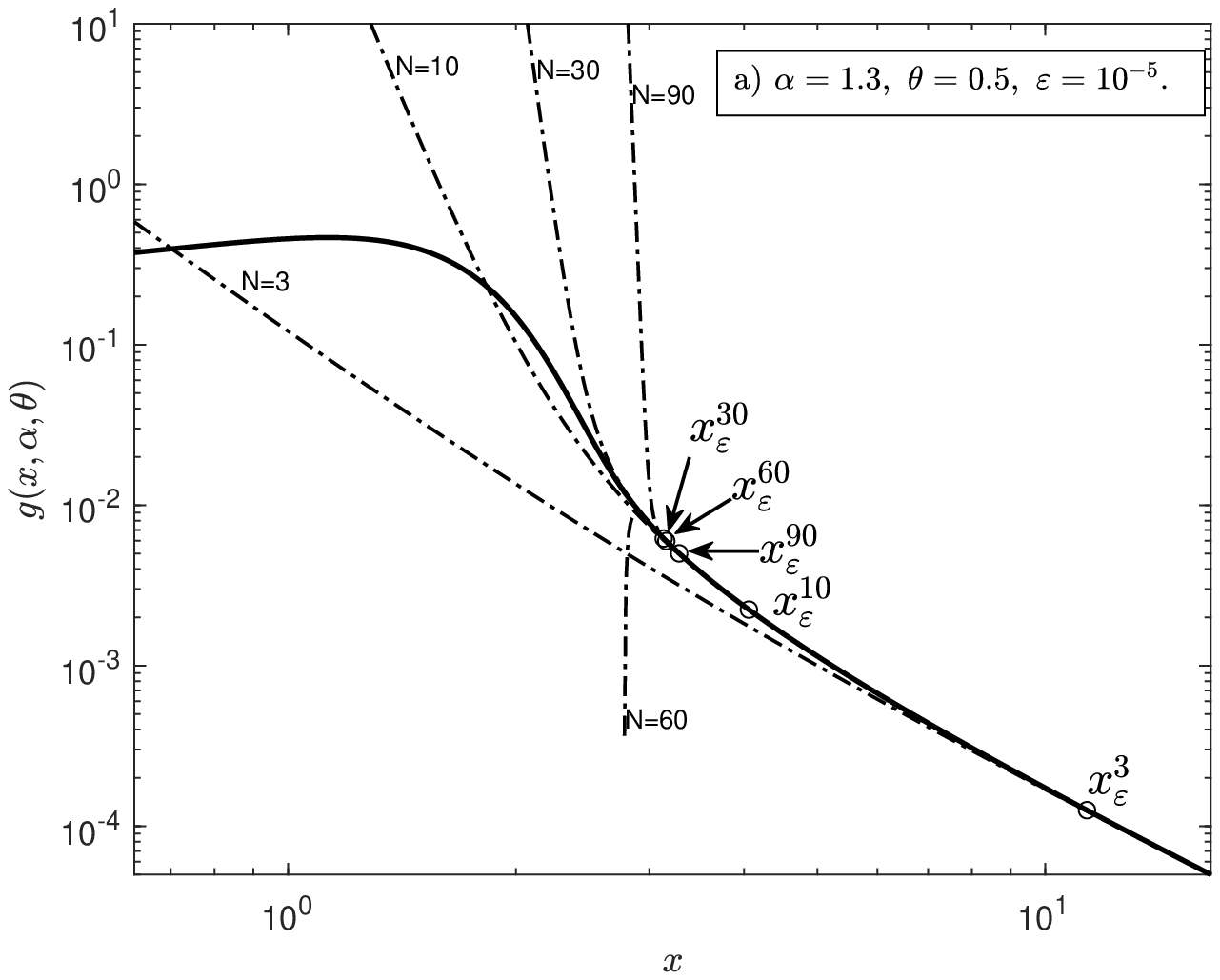}\hfill
  \includegraphics[width=0.475\textwidth]{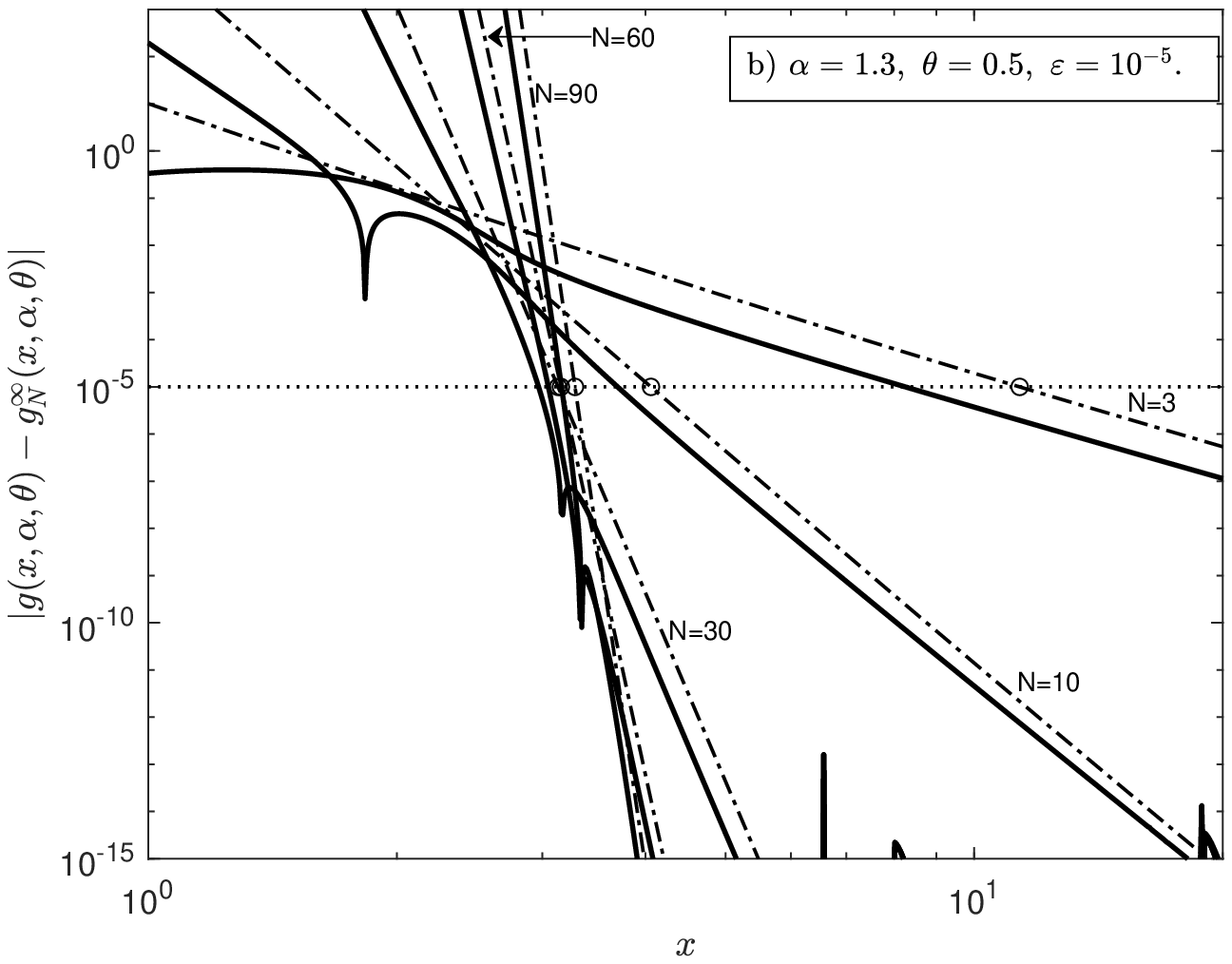}\\
  \caption{a) Probability density $g(x,\alpha,\theta)$ for the parameter values shown in the figure. The solid curve is the integral representation (\ref{eq:g(x)_int}), the dash-dotted curves are the formula (\ref{eq:g_NInf_calc}) for different values of the number of terms $N$ in the sum (\ref{eq:g_NInf}). The circles show the position of the threshold coordinate $x_\varepsilon^N$ for the corresponding values of $N$ and specified accuracy level $\varepsilon$. (b) Graph of the absolute error of calculating the density $g(x,\alpha,\theta)$ using the formula (\ref{eq:g_NInf_calc}). Solid curves are the exact value of the absolute error $|g(x,\alpha,\theta)-g_N^\infty(x,\alpha,\theta)|$, dash-dotted curves are the estimate of the remainder term (\ref{eq:R_NInf}), the dotted line is the specified accuracy level $\varepsilon$, the circles show the position of the threshold coordinate $x_\varepsilon^N$.}\label{fig:fig_a13}
\end{figure}

In the case $\alpha>1$ the behavior of the coordinate $x_\varepsilon^N$ changes with increasing $N$.  Figure~\ref{fig:fig_a13} shows the results of calculations for the case $\alpha=1.3$. Also, as in the previous cases, the position of the threshold coordinate is shown by circles (see, for example, Fig.~\ref{fig:fig_a13}a). As one can see from the figure, as the value of $N$ increases, the value of the threshold coordinate first decreases $x_\varepsilon^3>x_\varepsilon^{10}>x_\varepsilon^{30}$. However, a further increase in $N$ leads to an increase in the value of the threshold coordinate: $x_\varepsilon^{30}<x_\varepsilon^{60}<x_\varepsilon^{90}$. Such behavior of the threshold coordinate in the case of $\alpha>1$ is in full accordance with corollary~\ref{corol:pdfCoverg}.

Indeed, this corollary shows that in the case of $\alpha>1$ at $N\to\infty$ the series (\ref{eq:g_NInf}) is divergent. The reason for the divergence of the series is the presence of the multiplier $\Gamma(\alpha n+1)/\Gamma(n+1)$ in this series. As we can see at $\alpha>1$ this multiplier is more than 1 and as  $n$ increases this multiplier only increases. This fact leads to the divergence of this series at  $n\to\infty$. There is also a multiplier $x^{-\alpha n-1}$ in this sum. The competition between these two multipliers leads to the observed behavior of the threshold coordinate. Indeed, at small $n$ the addition of terms in the sum (\ref{eq:g_NInf}) first leads to a decrease in the coordinate $x_\varepsilon^N$ at which the equality is achieved $|g(x,\alpha,\theta)-g_N^\infty(x,\alpha,\theta)|=\varepsilon$. This is confirmed by the fact that  $x_\varepsilon^3>x_\varepsilon^{10}>x_\varepsilon^{30}$. However, a further increase in $n$ leads to an even greater increase in the value of the multiplier $\Gamma(\alpha n+1)/\Gamma(n+1)$. Therefore, to compensate for the increase in this multiplier and achieve the specified level of accuracy, it is necessary to reduce the value of the multiplier $x^{-\alpha n-1}$, which is achieved by increasing the value of the coordinate $x$. As a result, the coordinate $x_\varepsilon^N$ is shifted towards greater values. This is confirmed by the fact that $x_\varepsilon^{30}<x_\varepsilon^{60}<x_\varepsilon^{90}$.

\section{Calculation of the probability density at large $x$}

We will return to the question of calculating the probability density at large values of $x$. In the Introduction it was pointed out that the main approach to calculate the probability density is to use the integral representation (\ref{eq:g(x)_int}). In theory, this integral representation is valid for all values of the parameters $\alpha,\theta$ (except for the value $\alpha=1$) and all $x$. However, in practice, it is not always possible to calculate numerically the integral included in this integral representation. Problems arise at small and large values of the coordinate $x$.  The reason for the difficulties that arise is the behavior of the integrand in the formula (\ref{eq:g(x)_int}).

\begin{figure}
  \centering
  \includegraphics[width=0.8\textwidth]{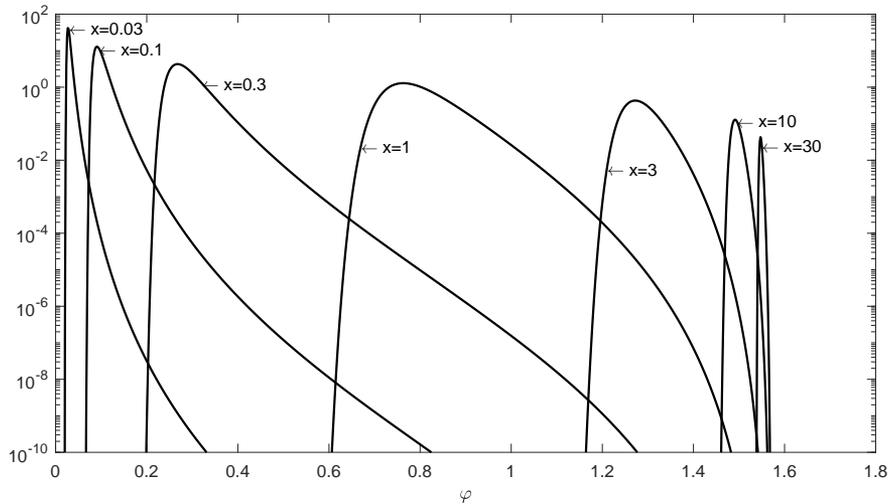}
  \caption{ The relationship between the integrand of the integral representation for the probability density (\ref{eq:g(x)_int}) and the integration variable $\varphi$. The figure shows graphs of the integrand for the values of the parameters $\alpha=1.1, \theta=0$ and the specified values of the coordinate $x$}\label{fig:pdfIntegrand}
\end{figure}

Figure~\ref{fig:pdfIntegrand} shows the graph of the relationship between the integrand in the formula (\ref{eq:g(x)_int}) and the integration variable $\varphi$ at different values of the coordinate $x$. One can see from the figure that at small and large values of $x$ the integrand turns into a function with a very narrow peak. As the value of $x$ decreases or increases, the width of this peak only decreases.  This behavior of the integrand leads to the fact that at very small or very large values of $x$ numerical integration algorithms cannot calculate the integral of this function. Therefore, in this case, it is expedient to use other approaches to calculate the probability density. The most appropriate approach is to use asymptotic expansions for the probability density. The problem of calculating the probability density in the case of $x\to0$ was considered in the article \cite{Saenko2022b}. This article deals with the problem of calculating the probability density at large values of $x$.

\begin{figure}
  \centering
  \includegraphics[width=0.48\textwidth]{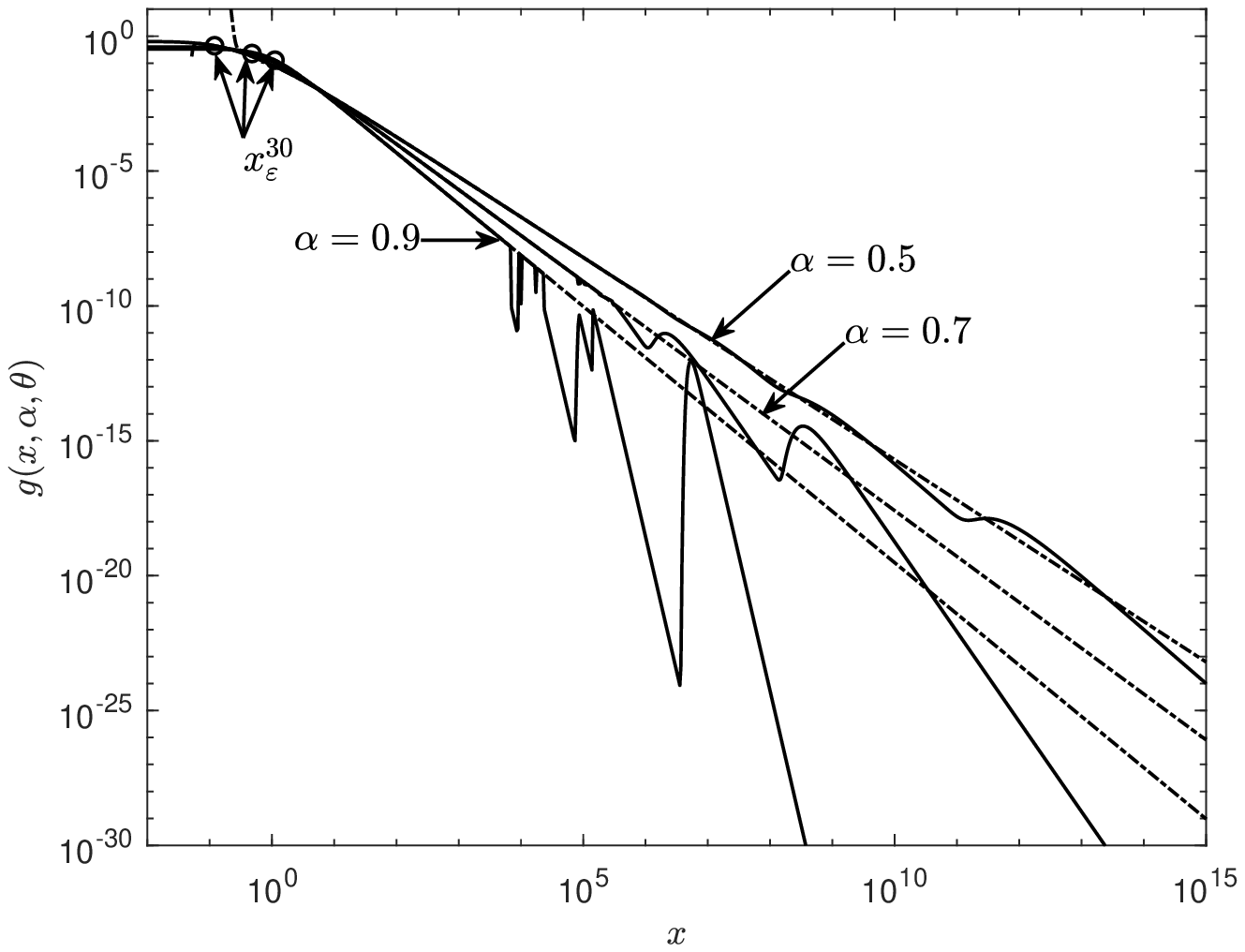}\hfill
  \includegraphics[width=0.48\textwidth]{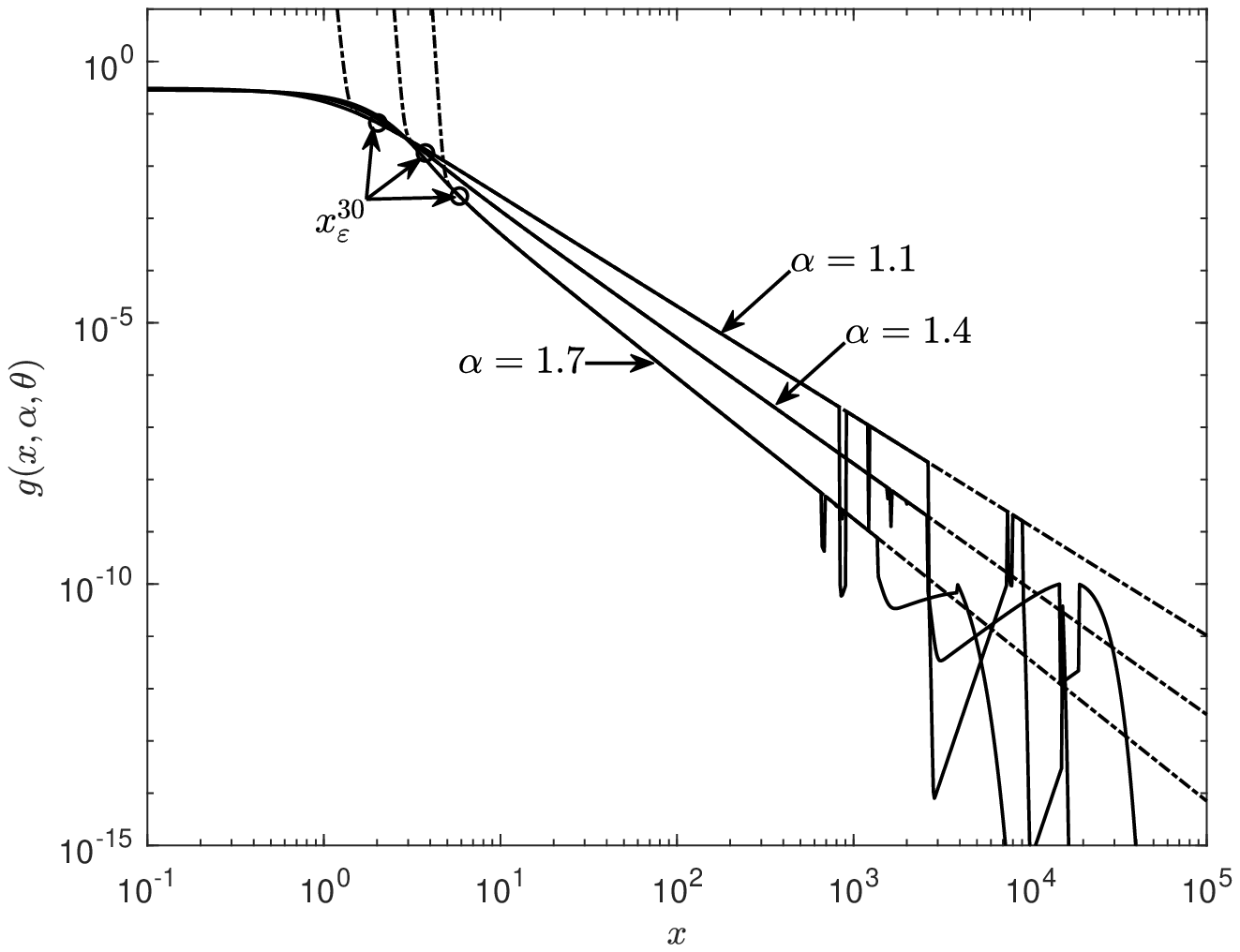}\\
  \caption{Probability density $g(x,\alpha,\theta)$ of a strictly stable law. The figure on the left is the case $\alpha<1$, the figure on the right is the case $\alpha>1$. The values of the index $\alpha$ are given in the figures. The solid curves are integral representation (\ref{eq:g(x)_int}), the dash-dotted curves are power series representation (\ref{eq:g_NInf_calc}), the circles are the position of the threshold coordinate $x_\varepsilon^N$ for each value $\alpha$, $N=30$ and $\varepsilon=10^{-5}$}\label{fig:pdf_x2inf}
\end{figure}

Figure~\ref{fig:pdf_x2inf} gives the results of calculating the probability density $g(x,\alpha,\theta)$ using the integral representation (\ref{eq:g(x)_int}) (solid curves) and using the formula (\ref{eq:g_NInf_calc}) (dash-dotted).  The values of the characteristic exponents $\alpha$ are given in the figures. To calculate the integral in the formula (\ref{eq:g(x)_int}) the Gauss-Kronrod algorithm was used.  To calculate the probability density using the formula (\ref{eq:g_NInf_calc}) the number of terms  $N=30$ was used. To find the value of the threshold coordinate $x_\varepsilon^N$ the value of the accuracy level $\varepsilon=10^{-5}$ was used. As we can see from the presented calculations, for large values of $x$ the numerical integration algorithm used is uncapable of calculating the integral in (\ref{eq:g(x)_int}) and starts producing incorrect values. One can also see from the figure that the value of the critical coordinate $x_{\mbox{\scriptsize cr}}$, at which the numerical integration algorithm begins to calculate the integral incorrectly, depends on the value of $\alpha$. For the value of $\alpha=0.5$ the value $x_{\mbox{\scriptsize cr}}\approx 10^7$,  for the value of $\alpha=0.7$ the value $x_{\mbox{\scriptsize cr}}\approx 10^5$,  for the value of $\alpha=0.9$ the value $x_{\mbox{\scriptsize cr}}\approx 10^4$. We can see that in the case $\alpha<1$ as the value of  $\alpha$ decreases, the value of $x_{\mbox{\scriptsize cr}}$ increases. In the case $\alpha>1$ from the figure we can see that for all presented $\alpha$ the value $x_{\mbox{\scriptsize cr}}\approx 4\cdot10^2 - 10^3$. Thus, at the values $x>x {\mbox{\scriptsize cr}}$ other methods should be used to calculate the probability density. The same problem exists for integral representations of the density of stable laws in other parameterizations of the characteristic function (see \cite{Nolan1997,Julian-Moreno2017,Royuela-del-Val2017,Ament2018}). To solve this problem, the authors used various numerical methods in these works, which make it possible to increase the accuracy of the calculation. However, these methods raise the accuracy of calculations, but do not solve the problem completely.

To solve the problem of calculating the probability density completely for large values of $x$ one can use theorem~\ref{theor:pdfExpansion} and, in particular, the formula (\ref{eq:g_NInf_calc}), which is the corollary of this theorem. As it was shown, the use of the formula (\ref{eq:g_NInf_calc}) guarantees that in the range of coordinates $x\geqslant x_\varepsilon^N$ the absolute error in calculating the probability density using this formula will not exceed the specified accuracy level $\varepsilon$. Moreover, as $x$ increases the absolute calculation error will only decrease. This follows directly from the estimate of the remainder term (\ref{eq:R_NInf}) and is clearly seen from the calculation results given in Figures~\ref{fig:fig_a07}b,~\ref{fig:fig_a1}b~and~\ref{fig:fig_a13}b.

In Fig.~\ref{fig:pdf_x2inf}  dash-dotted curves demonstrate the results of calculating the probability density using the formula (\ref{eq:g_NInf_calc}). The position of the threshold coordinate $x_\varepsilon^N$ for each value of $\alpha$ is shown by circles. The figure clearly shows from that in the domain $x_\varepsilon^N\leqslant x\leqslant x_{\mbox{\scriptsize cr}}$ the results of calculating the probability density using the integral representation (\ref{eq:g(x)_int}) and using the expansion in series (\ref{eq:g_NInf_calc}) coincide. At values $x>x_{\mbox{\scriptsize cr}}$ the numerical integration algorithm no longer makes it possible to obtain the correct value of the probability density using the integral representation (\ref{eq:g(x)_int}), while using the formulas (\ref{eq:g_NInf_calc}) does not lead to any difficulties in calculation. Consequently, in the domain $x>x_{\mbox{\scriptsize cr}}$ it is expedient to use the formula (\ref{eq:g_NInf_calc}) to calculate the probability density. Thus, the use of theorem~\ref{theor:pdfExpansion} and, in particular, the formula (\ref{eq:g_NInf_calc}) solves the problem of calculating the probability density for large values of $x$.

\section{Discussion}

The main problem in proving theorem~\ref{theor:pdfExpansion} was the calculation of the integral on the right side of the expression (\ref{eq:pdfExpan_tmp0}). Indeed, if to expand  the second exponent in a Taylor series in this expression, then we arrive at the integral
\begin{equation}\label{eq:auxInt}
\int_{0}^{\infty}e^{it}t^{\alpha k}dt.
\end{equation}
However, it is not possible to calculate this integral or bring this integral to one of the known integrals (in particular, to the Gamma function). The problem lies in the integration loop. The formula (\ref{eq:Gamma_intRepr1}) clearly shows this. This formula shows that in the definition of the Gamma function $\Gamma(z)=\int_{0}^{\infty} e^{-t} t^{\gamma-1}dt$ one can pass from the integration along the positive part of the real semiaxis to the integration along the ray coming from the origin of coordinates at the angle $\beta$. At the same time, the angle $\beta$ must satisfy the conditions: $-\frac{\pi}{2}<\beta<\frac{\pi}{2},\ \Re\gamma>0$ or $\beta\pm\frac{\pi}{2},\ 0<\Re\gamma<1$. As we can see, for the left part of the integral (\ref{eq:Gamma_intRepr1}) to take the form of the integral  (\ref{eq:auxInt}) it is necessary to set $\beta=-\frac{\pi}{2}$ and $c=1$. However, for this integral to correspond to the Gamma function in this case, the condition $0<\Re (\alpha k+1)<1$ must be satisfied, but this condition is not met at any $\alpha$. If we consider the condition $-\frac{\pi}{2}<\beta<\frac{\pi}{2},\ \Re\gamma>0$, then we see that it also  does not enable us to transform the integral (\ref{eq:auxInt}) to the Gamma function. Although in this case the condition $\Re (\alpha k+1)>0$ is satisfied but the range of the angle change $\beta$ does not include the point $\beta=-\frac{\pi}{2}$.  It is possible to show that in this case ($\Re\gamma>0$) at the points  $\beta=\pm\frac{\pi}{2}$ the integral of the left part (\ref{eq:Gamma_intRepr1}) will diverge. No transformation of the integral (\ref{eq:auxInt}) (in particular, by turning the integration contour) can transform this integral to the integral (\ref{eq:Gamma_intRepr1}).

It is possible to solve this problem if to take account that in the right part of the expression  (\ref{eq:pdfExpan_tmp0}) the relation $\tau/x\to0$ at $\tau\to\infty$ and $x\to\infty$. Taking account of this circumstance, it became possible to prove that if the conditions formulated in lemma~\ref{lem:LoopInt} are met, the equality (\ref{eq:loopInt_eq}) is satisfied. Lemma~\ref{lem:I_GammaR} is also devoted to the substantiation of the possibility of this equality. After proving the equality (\ref{eq:loopInt_eq}) it becomes possible to perform the passage from (\ref{eq:pdfExpan_tmp0}) to (\ref{eq:pdfExpan_tmp1}) and substantiate the possibility of choosing the argument $\varphi$ in the form (\ref{eq:varphi0}). After that, obtaining the expansion of the probability density at $x\to\infty$ in the form of a power series was not particularly difficult.

The main result of this article is formulated in theorem~\ref{theor:pdfExpansion}. In this theorem, in addition to the expansion of the probability density in a power series at $x\to\infty$, an estimate for the remainder term is also obtained. This estimate turned out to be very useful in solving the problem of calculating the probability density at large $x$. Indeed, the series (\ref{eq:g_NInf}) is valid at  $x\to\infty$, however, it does not give an answer to the question at which value of $x$ this series can be used to calculate the probability density. This question can be answered by evaluating the remainder term (\ref{eq:R_NInf}). This estimate makes it possible to introduce a threshold coordinate $x_\varepsilon^N$, at which the absolute error in calculating the probability density using the expansion (\ref{eq:g_NInf}) will not exceed the specified accuracy level $\varepsilon$. The threshold coordinate is found from the solution to the equation (\ref{eq:x_eps_eq}), in which the values $\alpha,\varepsilon$ and $N$ are given. Unfortunately, it is not possible to obtain the solution to the equation (\ref{eq:x_eps_eq}) in an explicit form. However, with the help of numerical methods, the solution to this equation can be found without any difficulties.

It should be noted that if to make an estimate of the remainder term (\ref{eq:R_NInf}) worse and represent it in the form
\begin{equation*}
  \left|R_N^{\infty}(x,\alpha,\theta)\right|\leqslant\frac{x^{-\alpha N-1}}{\pi N!}\left(\Gamma(\alpha N+1)+\Gamma(\alpha(N+1)+1)\right),\quad x>1,
\end{equation*}
then it gives an opportunity to find the explicit expression for the threshold coordinate
\begin{equation*}
  x_{\varepsilon,1}^N=\left(\frac{\Gamma(\alpha N+1)+\Gamma(\alpha(N+1)+1)}{\pi\varepsilon \Gamma(N+1)}\right)^{\frac{1}{\alpha N+1}}.
\end{equation*}
The convenience of the practical use of this expression is obvious: it is not necessary to resort to numerical methods every time and solve the equation (\ref{eq:x_eps_eq}) to find the value of the threshold coordinate. However, one always must pay for convenience, and the price here is the accuracy of determining the threshold coordinate. The estimate $x_{\varepsilon,1}^N$ is worse than the estimate $x_\varepsilon^N$ obtained from the solution to the equation (\ref{eq:x_eps_eq}).  In addition, it is valid only if the value $x_{\varepsilon,1}^N>1$. However, the choice of one or another formula always remains with a particular researcher.

\section{Conclusion}

Being the limiting distributions of sums of independent identically distributed random variables, stable laws are used in the description of various processes. They appear when describing the behavior of stock indices, the asymptotic distribution of the particle coordinate in the process of anomalous diffusion, when describing the distribution of gene expression levels, and so on. In this regard, there is often a necessity to calculate the probability density of this law. The main method for calculating the probability density is the use of integral representations. However, as shown in the article, this method is not always capable of giving the correct value of the probability density. Difficulties arise for small and large values of the coordinate $x$. They are connected not with the integral representation, the integral representation itself is valid for any $x$, but with numerical methods. Numerical integration algorithms turn out to be unable to calculate the integral in this formula for very small and very large values of $x$. Therefore, in these cases it is necessary to use other methods for calculating the probability density.

As it was mentioned in the Introduction, there are various methods to solve this problem, but all these methods are aimed at increasing the accuracy of calculations and do not solve the problem completely. In this article, to calculate the probability density for large $x$ the expansion of the probability density in a series at $x\to\infty$ is used. To achieve this goal, such a decomposition was obtained, as well as an estimate for the remainder term. The results were formulated in theorem~\ref{theor:pdfExpansion}. The convergence of this series was studied and it was shown that in the case of $\alpha<1$ the series was convergent at $N\to\infty$ for any $x$, in the case $\alpha=1$ the series was convergent at $N\to\infty$  and $|x|>1$, and in the case $\alpha>1$ the series was asymptotic at $x\to\infty$. These results were formulated in corollary~\ref{corol:pdfCoverg}. It should be noted that the results formulated in this corollary are not new. They generalize the known facts previously obtained by various authors regarding the convergence of the expansion of the probability density in a power series at $x\to\infty$ (see, for example, \cite{Zolotarev1986,Uchaikin1999,Feller1971_V2_en}). However, the study of the expansion of the probability density in the case $\alpha=1$ was carried out for the first time.  It was obtained that in this case the series was convergent at $|x|>1$. In addition, we managed to show that in this case at $N\to\infty$ the series converge to the probability density of the generalized Cauchy distribution (\ref{eq:g(x)_a=1}).

As it was noted, an estimate was obtained for the remainder term of the expansion of the probability density in a power series at $x\to\infty$. This estimate turned out to be very useful in solving the problem of calculating the probability density at large $x$. Using this estimate it was possible to introduce the threshold coordinate $x_\varepsilon^N$ and obtain the equation (\ref{eq:x_eps_eq}) for finding it. Threshold coordinates gives an opportunity to determine the domain of coordinates $x$ within which the absolute error of calculating the probability density using a power series  (\ref{eq:g_NInf}) at given value $\alpha$  and the number of summands  $N$ in the sum will not exceed the required level of accuracy $\varepsilon$. It enabled us to write a formula for calculation in the form  (\ref{eq:g_NInf_calc}). The calculations performed showed that when using this formula, the absolute error in calculating the probability density in the domain $|x|>x_\varepsilon^N$ does not exceed the required level of accuracy $\varepsilon$, and in reality is considerably less than this value. With the increase in $|x|$ the absolute calculation error only decreases.  This gives an opportunity to use the formula (\ref{eq:g_NInf_calc}) to calculate the probability density even  at those values of $x$, at which the integral representation  (\ref{eq:g(x)_int}) no longer makes it possible to obtain the correct result.

Indeed, the calculations performed showed that for the integral representation (\ref{eq:g(x)_int}) there was a critical values of the coordinate $x_{\mbox{\scriptsize cr}}$ at which numerical integration algorithms were no longer able to correctly calculate integral (see Fig.~\ref{fig:pdf_x2inf}). At the same time the use of the formula (\ref{eq:g_NInf_calc}) does not lead to any calculation difficulties. Thus, the use of this formula to calculate the probability density of a strictly stable law in the coordinate region $|x|>x_{\mbox{\scriptsize cr}}$, in which the use of the integral representation (\ref{eq:g(x)_int}) no longer leads to the correct result, solves the problem of calculating the probability density at large $x$.

As it was noted above, when using integral representations to calculate the probability density, difficulties in calculation arise both for large values of $x$, and small values of $x$. The problem of calculating the probability density at large $x$ was solved in this article, and the problem of calculating the probability density at small $x$ was solved in the article \cite{Saenko2022b}. To calculate the probability density at small $x$ in this article the expansion of the density at $x\to0$ and the estimate for the remainder term of this expansion were obtained.  Based on this estimate, an explicit expression for the threshold coordinate was obtained. Thus, if for large values of $x$ to use the formula (\ref{eq:g_NInf_calc}), at small values of $x$ to use the expansion from the article\cite{Saenko2022b}, and in the intermediate domain to use the integral representation (\ref{eq:g(x)_int}), then in this case we have an opportunity to calculate the probability density of a strictly stable law with a characteristic function (\ref{eq:CF_formC}) at any $x$.  Thus, the problem of calculating the probability density on the entire real line $x$ turns out to solved.

The similar problem exists when calculating the distribution function. In the article \cite{Saenko2022b} it was shown that for the integral representation of the distribution function of a strictly stable law, there are also problems with calculating the integral for small and large values of the coordinate $x$. In this article it was possible to solve the problem of calculating the distribution function at small values of  $x$. For this purpose, we used the expansion of the distribution function at $x\to0$. To solve the problem of calculating the distribution function completely, it remains to solve the problem of calculating the distribution function for large $x$. This problem can be solved if we get the expansion of the distribution function at $x\to\infty$. To obtain this expansion is possible if we use the result of theorem~\ref{theor:pdfExpansion}. However, this requires additional studies the results of which will be published in the nearest future.

\appendix
\section{Integral representations of the probability density}\label{sec:IntRepr}

To perform the inverse Fourier transform and obtain the distribution of the probability density, the following lemma is useful, which defines the inversion formula
\begin{lemma}\label{lem:Inverse}
The distribution of the probability density $g(x,\alpha,\theta)$ for any admissible set of parameters $(\alpha,\theta)$ and any $x$ can be obtained using inverse transformation formulas
\begin{equation}\label{eq:InverseFormula}
  g(x,\alpha,\theta)=\frac{1}{2\pi}\int_{-\infty}^{\infty}e^{-itx}\hat{g}(t,\alpha,\theta)dt=
  \left\{\begin{array}{c}
           \displaystyle\frac{1}{\pi}\Re\int_{0}^{\infty} e^{itx}\hat{g}(t,\alpha,-\theta)dt,  \\
           \displaystyle\frac{1}{\pi}\Re\int_{0}^{\infty} e^{-itx}\hat{g}(t,\alpha,\theta)dt.
         \end{array}\right.
\end{equation}
\end{lemma}
The proof of this lemma is given in the article \cite{Saenko2020b}.
Integral representation for the probability density of a strictly stable law with a characteristic function  (\ref{eq:CF_formC}) was obtained in the article\cite{Saenko2020b}. The following theorem was formulated and proved in this article.

\begin{theorem}\label{theorem:SSL_pdf}
The distribution density $g(x,\alpha,\theta)$ of a strictly stable law with the characteristic function (\ref{eq:CF_formC})  can be represented in the form
  \begin{enumerate}
    \item If $\alpha\neq1$ and $x\neq0$, for any values $|\theta|\leqslant\min(1,2/\alpha-1)$
    \begin{equation}\label{eq:g(x)_int}
      g(x,\alpha,\theta)=\frac{\alpha}{\pi|\alpha-1|}\int_{-\pi\theta^*/2}^{\pi/2}\exp\left\{-|x|^{\alpha/(\alpha-1)} U(\varphi,\alpha,\theta^*)\right\} U(\varphi,\alpha,\theta^*) |x|^{1/(\alpha-1)}d\varphi,
    \end{equation}
    where $\theta^*=\theta\sign{x}$ and
    \begin{equation}\label{eq:U_func}      U(\varphi,\alpha,\theta)=\left(\frac{\sin\left(\alpha\left(\varphi+\frac{\pi}{2}\theta\right)\right)}{\cos\varphi}\right)^{\alpha/(1-\alpha)} \frac{\cos\left(\varphi(1-\alpha)-\frac{\pi}{2}\alpha\theta\right)}{\cos\varphi}.
    \end{equation}
    \item If $x=0$, then for any $0<\alpha\leqslant2$ and $|\theta|\leqslant\min(1,2/\alpha-1)$
    \begin{equation*}
  g(0,\alpha,\theta)=\frac{1}{\pi}\cos\left(\frac{\pi\theta}{2}\right)\Gamma\left(\frac{1}{\alpha}+1\right)
    \end{equation*}
    \item If $\alpha=1$, then for any $|\theta|\leqslant1$ and any values $x$
    \begin{equation}\label{eq:g(x)_a=1}
      g(x,1,\theta)=\frac{\cos(\pi\theta/2)}{\pi(x^2-2x\sin(\pi\theta/2)+1)}.
    \end{equation}
  \end{enumerate}
\end{theorem}

The proof of the theorem is given in the article \cite{Saenko2020b}. As we can see, in the general case, the probability density is expressed in terms of a definite integral. However, in the particular case  $\alpha=1$ it is possible to perform the inverse Fourier transform of the characteristic function  (\ref{eq:CF_formC}) and express the probability density in terms of elementary functions. The formula (\ref{eq:g(x)_a=1}) generalizes the well-known Cauchy distribution for the case of an arbitrary asymmetry parameter $-1\leqslant\theta\leqslant1$. As we can see, at $\theta=0$ this distribution coincides with the Cauchy distribution and at $\theta=\pm1$ this distribution goes into a degenerate distribution at points $x\pm1$. This distribution first appeared in the book by V.M. Zolotarev \cite{Zolotarev1986}. Later this distribution was obtained and studied in the articles \cite{Saenko2020c,Saenko2020b}.
\bibliographystyle{elsarticle-num}
\bibliography{d:/bibliography/library}
\end{document}